\documentclass[a4paper]{amsart}

\usepackage{xcolor}
\definecolor{webgreen}{rgb}{0,.5,0}
\definecolor{webbrown}{rgb}{.6,0,0}
\definecolor{RoyalBlue}{cmyk}{1, 0.50, 0, 0}
\usepackage[colorlinks=true, breaklinks=true, urlcolor=webbrown, linkcolor=RoyalBlue, citecolor=webgreen,backref=page]{hyperref}

\usepackage{epsfig, graphicx, subfigure}
\usepackage{verbatim, setspace}
\usepackage{amsmath, amssymb}

\usepackage{newtxtext,newtxmath}

\usepackage{mathabx}

\newcommand{\T}		{\mathbb{T}}
\newcommand{\D}		{\mathbb{D}}
\newcommand{\R}		{\mathbb{R}}
\newcommand{\C}		{\mathbb{C}}
\newcommand{\N}		{\mathbb{N}}
\newcommand{\Z}		{\mathbb{Z}}

\newcommand{\supp}{\mathrm{supp}}

\renewcommand{\det}{\mathrm{det}}

\newcommand{\qasq}{\quad \text{as} \quad}
\newcommand{\qandq}{\quad \text{and} \quad}

\newcommand{\dd}{\mathrm{d}}
\newcommand{\ic}{\mathrm{i}}

\newcommand{\am}{\mathfrak{a}}

\newcommand{\z}{{\boldsymbol z}}
\newcommand{\s}{{\boldsymbol s}}
\newcommand{\e}{{\boldsymbol e}}
\newcommand{\x}{{\boldsymbol x}}

\newcommand{\RS}{\boldsymbol{\mathfrak S}}
\newcommand{\bd}{\boldsymbol{\Delta}}
\newcommand{\ualpha}{\boldsymbol\alpha}
\newcommand{\ubeta}{\boldsymbol\beta}

\newtheorem{theorem}{Theorem}[section]

\newtheorem{proposition}{Proposition}[section]

\numberwithin{equation}{section}

\begin{document}

\title{On \( L_\R^2\)-best rational approximants to Markov functions on several intervals}

\author{Maxim L. Yattselev}

\address{Department of Mathematical Sciences, Indiana University-Purdue University Indianapolis, 402~North Blackford Street, Indianapolis, IN 46202}

\email{\href{mailto:maxyatts@iupui.edu}{maxyatts@iupui.edu}}

\thanks{The research was supported in part by a grant from the Simons Foundation, CGM-706591.}

\subjclass[2000]{42C05, 41A20, 41A21}

\keywords{\(L^2\)-best rational approximation, strong asymptotics}

\maketitle

\begin{abstract}
Let \( f(z)=\int(z-x)^{-1}\dd\mu(x) \), where \( \mu \) is a Borel measure supported on several subintervals of \( (-1,1) \) with smooth Radon-Nikodym derivative. We study strong asymptotic behavior of the error of approximation \( (f-r_n)(z) \), where \( r_n(z) \) is the \( L_\R^2\)-best rational approximant to \( f(z) \) on the unit circle with \( n \) poles inside the unit disk. 
\end{abstract}

\section{\(L_\R^2\)-best Rational Approximants}

\subsection{Meromorphic Approximation Problem}
Let \( \D \) be the unit disk and \( \T \) be the unit circle. Denote by $L^2$ the space of square-summable functions on $\T$ and by $H^2$ the Hardy space of functions \( f(z) \) that are holomorphic in \( \D \) and satisfy
\[
\displaystyle \|f\|_2^2 := \sup_{0<r<1}\frac{1}{2\pi}\int_{\T}|f(r\tau)|^2|\dd\tau|<\infty.
\]
Every function in $H^2$ is uniquely determined by its non-tangential limit on $\T$, which necessarily belongs to \( L^2 \) and the $L^2$-norm of this trace is equal to the $H^2$-norm of the function. Thus, $H^2$ can be regarded as a closed subspace of $L^2$. We shall further denote by \( L^2_\R \) and \( H^2_\R \) the subspaces of  \( L^2 \) and \( H^2 \), respectively, consisting of functions with real Fourier coefficients, that is, functions satisfying \( f(\bar\tau)= \overline{f(\tau)} \) for \( \tau\in\T \).

Denote by \( \mathcal P_n \) the space of algebraic polynomials of degree at most $n$ with real coefficients and by \( \mathcal Q_n \) its subset consisting of monic polynomials with exactly $n$ zeroes in $\D$. Put
\[
\mathcal R_n := \left\{\frac{p(z)}{q(z)}=\frac{p_{n-1}z^{n-1}+p_{n-2}z^{n-2}+\cdots+p_0}{z^n+q_{n-1}z^{n-1}+\cdots+q_0}:~ p\in\mathcal P_{n-1},~q\in\mathcal Q_n\right\}
\]
and consider the following meromorphic approximation problem:  \emph{given $f\in L_\R^2$ and $n\in\N$, find $g_n\in  H_\R^2+\mathcal R_n$ such that $\|f-g_n\|_2=\inf_{g\in H_\R^2+\mathcal R_n}\|f-g\|_2$.} 

\subsection{Rational Approximation Problem}
The above problem can be reduced to a rational approximation problem. To this end, denote by $\bar H_\R^2$ the orthogonal complement of $H_\R^2$ in $L_\R^2$, $L_\R^2=H_\R^2\oplus\bar H_\R^2$, with respect to the standard scalar product. From the viewpoint of analytic function theory, $\bar H_\R^2$ can be regarded as a space of traces of functions holomorphic in $\{|z|>1\}$, having real Fourier coefficients and vanishing at infinity, and whose square-means on the concentric circles centered at zero (this time with radii greater then 1) are uniformly bounded above. The orthogonal decomposition $L_\R^2=H_\R^2\oplus\bar H_\R^2$ yields that
\[
\|f-g_n\|_2^2 = \|f^+ - g_n^+\|_2^2 + \|f^--g_n^-\|_2^2, 
\]
where $f=f^++f^-$, $g_n = g_n^++ g_n^-$, and $f^+,g_n^+\in H_\R^2$, $f^-,g_n^-\in\bar H_\R^2$. One can see that in order for $g_n$ to be a best approximant it is necessary that $g_n^+=f^+$. Moreover, one can readily check that $g_n^-\in\mathcal R_n$. Thus, an equivalent rational approximation problem can be stated as follows: \emph{given $f\in\bar H_\R^2$ and $n\in\N$, find $r_n\in \mathcal R_n$ such that $\|f-r_n\|_2=\inf_{r\in \mathcal R_n}\|f-r\|_2$.}

\subsection{Irreducible Critical Points}
$L_\R^2$-best rational approximants are a part of a larger class of \emph{critical points} in rational $\bar H_\R^2$-approximation. From the computational viewpoint, critical points are as important as best approximants  since a numerical search is more likely to yield a locally best rather than a best approximant. For a fixed $f\in\bar H^2_\R$, critical points can be defined as follows. Set
\[
\begin{array}{rll}
\Sigma_{f,n}:\mathcal P_{n-1}\times\mathcal Q_n &\to& [0,\infty) \\
(p,q) &\mapsto& \|f-p/q\|_2^2.
\end{array}
\]
In other words, $\Sigma_{f,n}$ is the squared error of approximation of $f$ by $r=p/q$ in $\mathcal R_n$. The cross-product $\mathcal P_{n-1}\times\mathcal Q_n$ is topologically identified with an open subset of $\R^{2n}$ with coordinates $p_j$ and $q_k$, $j,k\in\{0,\ldots,n-1\}$. Then a pair of polynomials $(p_c,q_c)\in\mathcal P_{n-1}\times\mathcal Q_n$, identified with a vector in $\R^{2n}$, is said to be a {\it critical pair of order} $n$, if all the partial derivatives of $\Sigma_{f,n}$ do vanish at $(p_c,q_c)$. Respectively, a rational function $r_c\in\mathcal R_n$ is a {\it critical point of order} $n$ if it can be written as the ratio $r_c=p_c/q_c$ of a critical pair $(p_c,q_c)$ in $\mathcal P_{n-1}\times\mathcal Q_n$. A particular example of a critical point is a {\it locally best approximant}. That is, a rational function $r_l=p_l/q_l$ associated with a pair $(p_l,q_l)\in\mathcal P_{n-1}\times\mathcal Q_n$ such that $\Sigma_{f,n}(p_l,q_l)\leq\Sigma_{f,n}(p,q)$ for all pairs $(p,q)$ in some neighborhood of $(p_l,q_l)$ in $\mathcal P_{n-1}\times\mathcal Q_n$. We call a critical point of order $n$ \emph{irreducible} if it belongs to $\mathcal R_n\setminus\mathcal R_{n-1}$. Best approximants, as well as local minima, are always irreducible critical points unless $f\in\mathcal R_{n-1}$. In general, there may be other critical points, reducible or irreducible, which are  saddles or maxima.

One of the most crucial features of the critical points is the fact that they are ``maximal'' rational interpolants. More precisely, if $f\in\bar H^2_\R$ and $r_n(z)$ is an irreducible critical point of order $n$, then \emph{$r_n(z)$ interpolates $f(z)$ at the reflection ($z\mapsto1/\bar z$) of each pole of $r_n(z)$ with order twice the multiplicity of the pole} \cite{Lev69}, which is the maximal number of interpolation conditions (i.e., $2n$) that can be imposed in general on a rational function of type $(n-1,n)$.

\subsection{Multipoint Pad\'e Approximants}

The just described property of irreducible critical points can also be expressed by saying that they are multipoint Pad\'e approximants. The latter are defined as follows. Let $f(z)$ be a function holomorphic and vanishing at infinity (the second condition is there for convenience only as functions in $\bar H^2_\R$ vanish at infinity by definition) and \( D \) be an unbounded domain in which \( f(z) \) is analytic or to which it possesses an analytic continuation, again, denoted by \( f(z) \). Let $\{E_n\}_{n\in\N}$ be a triangular scheme of points in \( D \), i.e., each $E_n$ consists of $2n$ not necessarily distinct nor necessarily finite points. Further, let $v_n(z)$ be the monic polynomial with zeroes at the finite points of $E_n$ (multiplicity of a zero is equal to the number of its occurrences in \( E_n \)). The \emph{$n$-th diagonal Pad\'e approximant of $f(z)$ associated with $E_n$} is a rational function $(p_n/q_n)(z)$ such that $\deg p_n\leq n$, $\deg q_n\leq n$, and $q_n\not\equiv0$, while
\begin{equation}
\label{Rn}
R_n(z) := \frac{q_n(z)f(z)-p_n(z)}{v_n(z)}=\mathcal{O}\left(1/z^{n+1}\right) \qasq z\to\infty,
\end{equation}
and is analytic in \( D \). Multipoint Pad\'e approximant always exists since the conditions for $p_n(z)$ and $q_n(z)$ amount to solving a system of $2n+1$ homogeneous linear equations with $2n+2$ unknown coefficients, no solution of which can be such that $q_n\equiv0$ (thus, we may assume that $q_n(z)$ is monic). Observe that given \( q_n(z) \), \( p_n(z) \) is uniquely defined. Indeed, if \( p(z) \) and \( p_*(z) \) were to correspond to the same \( q_n(z) \), the expression \( (p - p_*)(z)/v_n(z) \) would vanish at infinity with order at least \( n+1 \) while being finite at every zero of \( v_n(z) \), which is clearly impossible. Moreover, if the pairs \( p(z),q(z) \) and \( p_*(z),q_*(z) \) are solutions, then so is any linear combination \( (c_1p + c_2p_*)(z),(c_1q + c_2q_*)(z) \). Therefore, the solution corresponding to the monic denominator of the smallest degree is unique. In what follows, we understand that \(  p_n(z), q_n(z), R_n(z) \) come from this unique solution and write \( [n/n;E_n]_f(z):=(p_n/q_n)(z) \).

As mentioned above, irreducible critical points $r_n(z)=(p_n/q_n)(z)$ turn out to be multipoint Pad\'e approximants for which $v_n(z):=\varkappa_n\tilde q_n^2(z)$, where $\tilde q_n(z):=z^nq_n(1/z)$ is the reciprocal polynomial of $q_n(z)$ and \( \varkappa_n \) is the reciprocal of the leading coefficient of \( \tilde q_n^2(z) \).

\section{Markov Functions on One Interval}

\subsection{Markov Functions and Szeg\H{o} Measures}
Assume that
\begin{equation}
\label{markov}
f(z) = \int\frac{\dd\mu(x)}{z-x}
\end{equation}
for some finite positive compactly supported Borel measure \( \mu \) on the real line. Such functions, known as Markov functions, are well suited for asymptotic analysis of the behavior of their multipoint Pad\'e approximants because the denominator polynomials \( q_n(z) \) turn out to be orthogonal polynomials. More precisely, it holds that
\begin{equation}
\label{ortho}
\int x^m q_n(x)\frac{\dd\mu(x)}{v_n(x)} = 0, \quad m\in\{0,\ldots,n-1\},
\end{equation}
and the error of approximation can be represented as
\[
f(z)-\frac{p_n(z)}{q_n(z)} = \frac{v_n(z)}{q_n^2(z)}\int\frac{q_n^2(x)}{v_n(x)}\frac{\dd\mu(x)}{z-x}.
\]

In this section we shall assume that \( \mu \) is a Szeg\H{o} measure, that is, \( \supp(\mu)=[a,b] \) and
\begin{equation}
\label{decomp}
\dd\mu(x) = \frac{\dot{\mu}(x)\dd x}{\pi\sqrt{(x-a)(b-x)}} +  \dd\mu_s(x), 
\end{equation}
where \( \mu_s \) is singular to Lebesgue measure and \( \dot{\mu}(x) \) satisfies Szeg\H{o} condition \( \int\log\dot{\mu}(x)\dd x >-\infty \).

\subsection{Multipoint Pad\'e Approximants}
Set \( w(z):=w(z;a,b)=\sqrt{(z-a)(z-b)} \) to be holomorphic in \( \C\setminus[a,b] \) and normalized so that \( w(z)=z+\mathcal O(1) \) as \( z\to\infty \). Given a Szeg\H{o} measure, one can define its Szeg\H{o} function by
\[
S_{\dot\mu}(z) := \exp\left\{\frac{w(z)}{2\pi\ic}\int_a^b\frac{\log\dot{\mu}(x)}{x-z}\frac{\dd x}{w_+(x)}\right\}, \quad z\in\overline\C\setminus[a,b],
\]
where \( F_\pm(x) := \lim_{y\to0} F(x\pm\ic y) \), \( x\in(a,b) \), for any function \( F(z) \) holomorphic off \( [a,b] \). This is a holomorphic and non-vanishing function in \( \overline\C\setminus[a,b] \) whose traces exist almost everywhere on \( [a,b] \) and satisfy \( |S_{\dot\mu\pm}(x)|^2=\dot{\mu}(x)\). Now, let \( \psi(z) \) be the conformal map of \( \overline\C\setminus[a,b] \) onto \( \D \) with \( \psi(\infty)=0 \) and \( \psi^\prime(\infty)>0 \). It is given by
\[
\psi(z) = \frac2{b-a}\left(z-\frac{b+a}2 - w(z)\right).
\]
We shall assume that the interpolation sets \( E_n \) are conjugate-symmetric (i.e., if \( e\in E_n \), then \( \bar e\in E_n \)) and lie sufficiently far away from \( [a,b] \) in the sense that
\begin{equation}
\label{cond-En}
\lim_{n\to\infty}\sum_{e\in E_n}\big(1-|\psi(e)|\big) = \infty.
\end{equation}
The above condition is always satisfied if there exists a neighborhood of \( [a,b] \) devoid of elements of all \( E_n \). 

\begin{theorem}
\label{thm:M1-MP}
Let \( f(z) \) be given by \eqref{markov} for a Szeg\H{o} measure \( \mu \), supported on \( [a,b] \), and \( p_n(z)/q_n(z) \) be the multipoint Pad\'e approximant of \( f(z) \) associated with \( E_n \), where the interpolation sets \( E_n \) are conjugate-symmetric and satisfy \eqref{cond-En}. Then it holds that
\[
f(z) - \frac {p_n(z)}{q_n(z)} = (2+o(1))\frac{S_{\dot\mu}^2(z)}{w(z)}\prod_{e\in E_n}\frac{\psi(z)-\psi(e)}{1-\psi(z)\overline{\psi(e)}}
\]
as \( n\to \infty \) locally uniformly in \( \overline\C\setminus[a,b] \) (condition \eqref{cond-En} ensures that the products above converge to zero locally uniformly in \( \overline\C\setminus[a,b] \)).
\end{theorem}

As stated this theorem is taken from \cite[Theorem~2]{St00}, with an additional admissibility condition it was proven in \cite{CYL98}, and the case of absolutely continuous measures and the interpolation points separated from \( [a,b] \) was considered in \cite{Totik}.

\subsection{Irreducible Critical Points}

Let now \( [a,b]\subset(-1,1) \), in which case \( f\in \bar H^2_\R \). Let \( r_n(z)=p_n(z)/q_n(z) \) be an irreducible critical point of order \( n \) in rational \( \bar H_\R^2 \)-approximation of \( f(z) \). Since the denominator polynomials satisfy \eqref{ortho} with \( v_n(z)=\varkappa_n\tilde q_n^2(z) \), all their zeroes belong to \( [a,b] \) and therefore the zeroes of \( \tilde q_n(z) \) belong to \( [a,b]^{-1}:=\{x:x^{-1}\in[a,b]\} \). Then one can reformulate Theorem~\ref{thm:M1-MP} in a way more suitable for \( \bar H_\R^2 \)-approximants. 

To this end, let \( w(z)=w(z;a,b) \) be as above and set \( \tilde w(z) := zw(1/z) \). Observe that \( \tilde w(z) \) is holomorphic in \( \C\setminus[a,b]^{-1} \) and normalized so that \( w(0)=1 \). Put
\[
\dd\omega_{[a,b],\T}(x) := \frac{1-ab}{2K}\frac{\dd x}{|(w\tilde w)(x)|},  \quad x\in[a,b],
\]
where the constant \( K \) turns \( \omega_{[a,b],\T} \) into a probability measure and is, in fact, the complete elliptic integral of the first kind with modulus \( (b-a)/(1-ab) \). The measure \( \omega_{[a,b],\T} \) can be interpreted from the point of view of potential theory as the condenser equilibrium distribution on \( [a,b] \) of the condenser \( ([a,b],\T) \), see \cite{BStW01}. Define
\[
G_{\dot\mu} := \exp\left\{\int\log\dot\mu(x)\dd\omega_{[a,b],\T}(x)\right\}.
\]
It holds that \( G_{\dot\mu}>0 \) if and only if \( \mu \) is a Szeg\H{o} measure on \( [a,b] \), see \cite[Lemma~2.37]{BStW01}. Put
\[
D_{\dot\mu}(z) := \exp\left\{\frac{(w\tilde w)(z)}{2\pi\ic} \int_a^b \frac{1-2xz+x^2}{(x-z)(1-xz)}\log\left(\frac{\dot\mu(x)}{G_{\dot\mu}}\right)\frac{\dd x}{(w_+\tilde w)(x)}\right\}
\]
for \( z\in\overline\C\setminus\big([a,b]\cup[a,b]^{-1}\big) \). The function \( D_{\dot\mu}(z) \) is non-vanishing and holomorphic in \( \overline\C\setminus\big([a,b]\cup[a,b]^{-1}\big) \), its argument has zero increment along \( \T \) and \( |D_{\dot\mu}(\tau)|\equiv 1 \) for \( \tau\in \T \). Moreover, its traces exist almost everywhere on \( [a,b]\cup[a,b]^{-1} \) and satisfy \( G_{\dot\mu}|D_{\dot\mu\pm}(x)|^2 = \dot\mu(x) \), \( x\in[a,b] \), and \( G_{\dot\mu}|D_{\dot\mu\pm}(x)|^{-2} = \dot\mu(1/x) \), \( x\in[a,b]^{-1} \), see \cite[Lemma~2.40]{BStW01}. Further, let
\[
\varphi(z) := \exp\left\{ \pi\frac{1-ab}{2K}\int_1^z\frac{\dd s}{(w\tilde w)(s)} \right\},
\]
where the path of integration belongs to \( \overline\C\setminus\big([a,b]\cup[a,b]^{-1}\big) \). This function is in fact the conformal map of the ring domain  \( \overline\C\setminus\big([a,b]\cup[a,b]^{-1}\big) \) onto the annulus \( \{z:~\rho<|z|<1/\rho\}\), where \( \rho := \varphi(b) \). The following theorem was proven in \cite[Theorem~8]{BStW01}.

\begin{theorem}
\label{thm:M1-CP}
Let \( f(z) \) be given by \eqref{markov}, for a Szeg\H{o} measure \(\mu \) on \( [a,b]\subset(-1,1) \). Further, let \( \{r_n(z)\} \) be a sequence of irreducible critical points in rational \( \bar H_\R^2 \)-approximation of \( f(z) \). Then it holds that
\[
f(z)-r_n(z) = \big(2G_{\dot\mu}+o(1)\big)\frac{D_{\dot\mu}^2(z)}{w(z)}\left(\frac\rho{\varphi(z)}\right)^{2n}
\]
as \( n\to \infty \) locally uniformly in \( \overline\C\setminus\big([a,b]\cup[a,b]^{-1}\big) \). 
\end{theorem}

\section{Pad\'e Approximation of Markov Functions On Several Intervals}

In this section we shall explain how results from \cite{Y18} specialize to the case of Markov functions on several intervals. Part of the goal of this section is to provide as many explicit formulae as possible (most of them were not presented in \cite{Y18}). 

\subsection{Measures of Orthogonality}
Assume now that \( \supp(\mu) = \cup_{i=1}^{g+1}[a_i,b_i] =: \Delta \), where \( g\geq1 \) and, for convenience, we assume that \( b_i<a_{i+1} \), \( i\in\{1,\ldots,g\} \). Let 
\begin{equation}
\label{w}
w(z) := \sqrt{(z-a_1)(z-b_1)\cdots(z-a_{g+1})(z-b_{g+1})}
\end{equation}
be holomorphic in \( \overline\C\setminus\Delta \) and normalized so that \( w(z)=z^{g+1} + \mathcal O(z^g) \) as \( z\to\infty \). In this section we shall suppose that \eqref{decomp} is replaced by more stringent assumption 
\begin{equation}
\label{smooth-meas1}
\dd\mu(x) = -\frac1{\pi\ic}\frac{\rho(x)\dd x}{w_+(x)}, \quad x\in\Delta,
\end{equation}
where \( \rho(x) \) is a real-valued, smooth, see \eqref{sobolev} further below, and non-vanishing function on \( \Delta \) whose sign distribution is such that \( \mu \) is a positive measure. To capture the positivity, choose \( m(x) \) to be a monic polynomial of degree \( g \) with exactly one zero in each gap \( (b_k,a_{k+1}) \), \( k\in\{1,\ldots,g\} \). Observe also that \( -\ic w_+(x) = (-1)^{g+1-k}|w(x)| \) for \( x\in[a_k,b_k] \), \( k\in\{1,\ldots,g+1\} \). Hence, if we write
\begin{equation}
\label{smooth-meas2}
\rho(x) =: \dot\mu(x) m(x), \quad x\in\Delta, 
\end{equation}
then \( \dot\mu(x) \) is a positive function on \( \Delta \) from which we require that its logarithm belongs to a fractional Sobolev space:
\begin{equation}
\label{sobolev}
\iint_{\Delta\times\Delta}\left|\frac{\log\dot\mu(x)-\log\dot\mu(y)}{x-y}\right|^p \dd x\dd y<\infty
\end{equation}
for some \( p>4 \) (then \( \log\dot\mu(x) \) is H\"older continuous with exponent at least \( 1-2/p>1/2 \)). Notice that \( \rho(x) \) also admits a continuous branch of the logarithm on \( \Delta \) that satisfies \eqref{sobolev} (this is the condition used in \cite{Y18}). A particular choice of \( m(x) \), on which \( \dot\mu(x) \) does depend, is not important to us. In fact, we could have chosen \( m(x) \) simply to be \( 1 \) or \( -1 \) on the intervals comprising \( \Delta \). However, in the author's opinion, some of the formulae further below would have been less elegant in this case.

\subsection{Riemann Surface and its Holomorphic Differentials}

Let \( \RS \) be a Riemann surface realized as two copies of \( \overline\C\setminus\Delta \), say \( \boldsymbol D \) and \( \boldsymbol D^* \) glued crosswise along \( \Delta \). Formally, it can be defined as 
\[ 
\RS=\left\{\z=(z,w):w^2=\prod_{i=1}^{g+1}(z-a_i)(z-b_i)\right\}.
\]
Of course, it holds that \( \RS = \boldsymbol D \cup \boldsymbol\Delta \cup \boldsymbol D^* \), where \( \boldsymbol \Delta := \pi^{-1}(\Delta) \) and \( \pi(\z) = z \) is the natural projection. We denote by \( \cdot^* \) a conformal involution defined by \( \z^* = (z,-w) \) for \( \z=(z,w) \). We choose a homology basis \( \{\ualpha_i,\ubeta_i\}_{i=1}^g\) on \( \RS \) in the following way: \(\ualpha_i:= \pi^{-1}([b_i,a_{i+1}]) \) and is oriented away from \( \boldsymbol b_i \) on \( \boldsymbol D \) while \( \ubeta_i\subset\boldsymbol D\cup\bd \) is such that \( \pi(\ubeta_i) \) is a convex Jordan curve that passes through \( a_1 \) and \( b_i \) and is oriented counter-clockwise (we also assume that \( \pi(\ubeta_i)\setminus\{a_1\} \) is contained in the interior of \( \pi(\ubeta_{i+1})\)). 

Let \( \mathsf V \) be the following matrix:
\[
\mathsf V := \left[\int_{b_i}^{a_{i+1}}\frac{x^l\dd x}{w(x)}\right]_{i=1,l=0}^{g,g-1},
\]
where \( i \) is the row index and \( l \) is the column one. It is straightforward to see that
\[
\det(\mathsf V) = \int_{b_1}^{a_2}\cdots\int_{b_g}^{a_{g+1}}\frac{V(x_1,\ldots,x_g)}{w(x_1)\cdots w(x_g)}\dd x_g\cdots \dd x_1 \neq 0,
\]
where \( V(x_1,\ldots,x_g) \) is the Vandermonde determinant and the inequality \( \det(\mathsf V) \neq 0 \) follows from the obvious fact that all the gaps \( (b_k,a_{k+1}) \) are disjoint. Hence, there exist polynomials \( l_i(x) \), \( \deg(l_i)=g-1 \), such that
\begin{equation}
\label{li}
\int_{b_k}^{a_{k+1}}\frac{l_i(x)\dd x}{w(x)} = \delta_{ki}, \quad i,k\in\{1,\ldots,g\},
\end{equation}
where \( \delta_{ki} \) is the usual Kronecker symbol. Indeed, the coefficients of these polynomials are the columns of \( \mathsf V^{-1} \). Moreover, since these polynomials have degree \( g-1 \) and real coefficients, \eqref{li} yields that each \( l_i(x) \) has exactly one zero in each gap \( (b_k,a_{k+1}) \) for \( k\neq i \). 

Observe also that the differentials
\begin{equation}
\label{hol-diff}
\mathcal H_i(\z) := \left\{
\begin{array}{rl} 
(l_i(z)\dd z)/(2w(z)), & \z\in\boldsymbol D, \smallskip \\
-\mathcal H_i(\z^*), & \z\in\boldsymbol D^*, 
\end{array}
\right.
\end{equation}
are holomorphic on \( \RS \) and normalized so that \( \oint_{\ualpha_k}\mathcal H_i = \delta_{ki} \) (since \( w_+(x)=-w_-(x) \) on \( \Delta \), these differentials do holomorphically extend across \( \bd \)). It is a well-known fact of the theory of compact Riemann surfaces, see \cite[Section~III.2.7]{FarkasKra}, that the integrals of the normalized holomorphic differentials over the \(\ubeta\)-cycles of the chosen homology basis form a symmetric matrix with positive definite imaginary part:
\begin{equation}
\label{B}
\mathsf B := \big[\mathsf B_{kj}\big]_{k,j=1}^g, \quad \mathsf B_{kj} := \oint_{\ubeta_k}\mathcal H_j = -\sum_{i=1}^k\int_{a_i}^{b_i}\frac{l_j(x)\dd x}{w_+(x)},
\end{equation}
where to get the second representation for \( \mathsf B_{kj} \) we used holomorphy of the integrands to deform \( \ubeta_k \) to \(\pi^{-1}([a_1,b_k])\cap(\boldsymbol D\cup\bd)\) and the fact that the integrals over the gaps cancel each other out. This  representation shows that all the entries of \( \mathsf B \) are purely imaginary.

\subsection{Third Kind Differentials}
Let \( m_\infty(z) \) be a polynomial of degree \( g \) such that
\begin{equation}
\label{m-infty}
\frac1{2\pi\ic}\int_{|z|=\rho} \frac{m_\infty(s)\dd s}{w(s)}=-1 \qandq \int_{b_i}^{a_{i+1}}\frac{m_\infty(x)\dd x}{w(x)} = 0, \;\; i\in\{1,\ldots,g\},
\end{equation}
where the circle \( |z|=\rho \) is positively oriented and contains \( \Delta \) in its interior. Using polynomials \( l_i(z) \) from \eqref{li}, it can be readily checked that
\[
m_\infty(x) = -x^g + \sum_{i=1}^g\left(\int_{b_i}^{a_{i+1}}\frac{y^g\dd y}{w(y)}\right)l_i(x).
\]
Clearly, the polynomial \( m_\infty(x) \) has real coefficients and therefore has exactly one simple zero in each gap \( (b_i,a_{i+1}) \). Thus, the measure \( \dd\omega_\Delta(x) = (m_\infty(x)\dd x)/(\pi\ic w_+(x)) \) is a positive probability measure on \( \Delta \). In fact, from the point of view of potential theory, \( \omega_\Delta \) can be interpreted as the logarithmic equilibrium distribution on \( \Delta \) (in particular, if one chooses \( m(x)=m_\infty(x) \) in \eqref{smooth-meas2}, then \( \dot\mu(x) \) again can be interpreted as the Radon-Nikodym derivative with respect to the logarithmic equilibrium distribution).

More generally, given \( e\in\C\setminus[a_1,b_{g+1}] \), let \( m_e(z) \) be a polynomial of degree \( g \) such that
\begin{equation}
\label{m-e}
\frac1{2\pi\ic}\int_{|z-e|=\rho} \frac{m_e(s)}{s-e}\frac{\dd s}{w(s)}=1 \qandq \int_{b_i}^{a_{i+1}}\frac{m_e(x)}{x-e}\frac{\dd x}{w(x)} = 0, \;\; i\in\{1,\ldots,g\},
\end{equation}
where the circle \( |z-e|=\rho \) is positively oriented and contains \( \Delta \) in its exterior. Similarly to \( m_\infty(x) \), one can check that
\[
m_e(x) = c_e\left(1-\sum_{i=1}^g\left(\int_{b_i}^{a_{i+1}}\frac1{y-e}\frac{\dd y}{w(y)}\right)(x-e)l_i(x)\right),
\] 
where the normalizing constant \( c_e \) is chosen so that the first condition in \eqref{m-e} is fulfilled. Notice that \( m_e(x) \) is a polynomial with real coefficients when \( e \)  is real (in particular, it has exactly one zero in each gap \( (b_i,a_{i+1})\)). In this situation the corresponding measure \( (m_e(x)\dd x)/(\pi\ic(x-e) w_+(x)) \) is a positive probability measure on \( \Delta \) and can be interpreted as the weighted equilibrium distribution in the field generated by a single unit charge placed at \( e \), or equivalently as the balayage of the Dirac mass at \( e \) to \( \Delta \).  This is no longer the case when \( \Im e\neq 0 \). However, since \( m_{\bar e}(z) = \overline{m_e(\bar z)}\), it holds that \( (z-\bar e)m_e(z) + (z-e)m_{\bar e}(z) \) is a polynomial with real coefficients as well (again, it must have exactly one zero in each gap). In this case the average of densities corresponding to \( e \) and \( \bar e \) produces a positive probability measure with a similar potential-theoretic interpretation.

Let \( \Omega_e(\z) \) be the normalized (that is, having zero periods on the \( \ualpha \)-cycles) third kind differential on \( \RS \) with two simple poles at \( \boldsymbol e,\boldsymbol e^* \), \( \boldsymbol e\in\boldsymbol D \), with respective residues \( 1 \) and \( -1 \), which is otherwise holomorphic. It can be readily checked that
\[
\Omega_e(\z) = \left\{
\begin{array}{rl} 
 (m_e(z)\dd z)/((z-e)w(z)), & \z\in\boldsymbol D, \smallskip \\
-\Omega_e(\z^*), & \z\in\boldsymbol D^*. 
\end{array}
\right.
\]
When \( e \) is real (including \( e=\infty\), in which case \( m_e(z)/(z-e) \) above is replaced by \( m_\infty(z) \)), \( \Omega_e(\z) \) can be considered as normalized to have purely imaginary periods on all of the cycles of the homology basis due to the second requirements in \eqref{m-infty} and \eqref{m-e}. This is no longer true when \( \Im e\neq 0 \), however, the differential \( \Omega_e(\z) + \Omega_{\bar e}(\z) \) can also be seen as normalized to have purely imaginary periods.

\subsection{Blaschke-type Functions}
Let \( E_n \) be a conjugate-symmetric interpolation set in \( \overline\C\setminus[a_1,b_{g+1}] \) (as usual, it consists of \( 2n \) not necessarily finite nor distinct elements).  Let
\begin{equation}
\label{psin}
\psi_n(z) := \exp\left\{\int_{b_{g+1}}^z\left(\sum_{e\in E_n,|e|<\infty}\frac{m_e(s)}{s-e} + \sum_{e\in E_n,|e|=\infty}m_\infty(s)\right)\frac{\dd s}{w(s)}\right\},
\end{equation}
where the path of integration lies entirely in \( \C\setminus(-\infty,b_{g+1}) \). Furthermore, set 
\begin{equation}
\label{omegank}
\omega_{n,k} := \mathrm{fr}\left\{-\frac1{2\pi\ic}\sum_{i=1}^k \int_{a_i}^{b_i}\left(\sum_{e\in E_n,|e|<\infty}\frac{m_e(x)}{x-e} + \sum_{e\in E_n,|e|=\infty}m_\infty(x)\right)\frac{\dd x}{w_+(x)}\right\}
\end{equation}
\( k\in\{1,\ldots,g\} \), where \( \mathrm{fr}\{x\}\in[0,1) \) is such that \( x-\mathrm{fr}\{x\} \in \Z \). It follows from the conjugate-symmetry of \( E_n \) and the discussion after \eqref{m-e} that these constants are real.

\begin{proposition}
\label{prop:psin}
The function \( \psi_n(z) \) is analytic in \( \overline\C\setminus[a_1,b_{g+1}] \) and has a zero at each \( e\in E_n \) of order equal to the multiplicity of \( e \) in \( E_n \). It holds that \( |\psi_n(z)|<1 \) in \( \overline\C\setminus\Delta \) and
\begin{equation}
\label{psin-bdry}
\left\{
\begin{array}{ll}
\psi_{n+}(x) = \psi_{n-}(x)e^{-4\pi\ic\omega_{n,k}}, & x\in(b_k,a_{k+1}), \quad k\in\{1,\ldots,g\},\medskip \\
|\psi_{n\pm}(x)| \equiv 1, & x\in\Delta.
\end{array}
\right.
\end{equation}

\end{proposition}
\begin{proof}
The integrand in \eqref{psin} behaves like \( -ks^{-1} + \mathcal O(s^{-2}) \) as \( s\to\infty \), where \( k \) is the multiplicity of \( \infty \) in \( E_n \). Therefore, the integral of the integrand is equal to an integer multiple of \( 2\pi\ic \) on any closed curve encircling \( [a_1,b_{g+1}] \), which implies analyticity of \( \psi_n(z) \) in \( \overline\C\setminus[a_1,b_{g+1}] \). Vanishing of \( \psi_n(z) \) at \( e\in E_n \) follows from the first requirements in \eqref{m-infty} and \eqref{m-e}. Since the integrand in \eqref{psin} is real on \( \R\setminus\Delta \), \eqref{psin-bdry} follows from conjugate symmetry of \( E_n \) and the second requirements in \eqref{m-infty} and \eqref{m-e}. As the function \( \log|\psi_n(z)| \) is subharmonic in \( \overline\C\setminus\Delta \) and is identically zero on \( \Delta \), the conclusion \( |\psi_n(z)|<1 \) follows from the maximum principle for subharmonic functions \cite[Theorem~2.3.1]{Ransford}.
\end{proof}

For the future comparison with \cite{Y18}, let us point out that a function
\begin{equation}
\label{Sn}
S_n(\z) := \exp\left\{\int_{\boldsymbol b_{g+1}}^\z G_n\right\}\left\{ \begin{array}{ll} 1, & \z\in\boldsymbol D, \smallskip \\ v_n^{-1}(z), & \z\in\boldsymbol D^*,\end{array}\right.
\end{equation}
was defined there, where, as before, \( v_n(z)=\prod_{e\in E_n,|e|<\infty}(z-e) \), \( \boldsymbol a_i,\boldsymbol b_i \) are the ramification point of \( \RS \) with the respective natural projections \( a_i,b_i \), \( i\in\{1,\ldots,g+1\} \), and \( G_n(\z) \) is a meromorphic differential on \( \RS \) given by
\[
G_n(\z) := \frac12\sum_{e\in E_n,|e|<\infty}\left(\frac{\dd z}{z-e}-\Omega_e(\z)\right) - \frac12\sum_{e\in E_n,|e|=\infty}\Omega_\infty(\z).
\]
This function is holomorphic and non-vanishing in \( \pi^{-1}(\C\setminus[a_1,b_{g+1}]) \) with a pole and a zero of order \( n \) lying on top of infinity in \( \boldsymbol D\) and \( \boldsymbol D^* \), respectively. It then follows that
\begin{equation}
\label{psin-sn}
\psi_n(z) = v_n(z)S_n(\z^*)S_n^{-1}(\z), \quad \z\in\boldsymbol D.
\end{equation}
As far as the boundary values of \( S_n(\z) \) are concerned, it holds that
\begin{equation}
\label{Sn-jump}
S_{n+}(\s)=S_{n-}(\s)\left\{
\begin{array}{rl} 
v_n(s), & \s\in\bd, \smallskip \\
e^{2\pi\ic\omega_{n,k}}, & \s\in\ualpha_k, \;\; k\in\{1,\ldots,g\}.
\end{array}
\right.
\end{equation}
There exists an alternative construction of the functions \( S_n(\z) \). It  will be presented further below in Section~\ref{sec:3.8} for the sake of completeness of the exposition. There, rather than using third kind differentials, we shall use Riemann theta functions.

\subsection{Szeg\H{o} Functions}

Let \( p(z) \) be a monic polynomial of degree \( d\leq g \) with simple zeroes, say \( z_1,\ldots,z_d \) (of course, when \( d=0 \), there are no zeroes). Assume that the zeroes of \( p(z) \) do not belong to \( \Delta \). Given \( \mu \) as in \eqref{smooth-meas1}--\eqref{smooth-meas2}, let
\begin{equation}
\label{szego}
S_{\dot\mu}(z) := \exp\left\{\frac1{2\pi\ic}\frac{w(z)}{p(z)}\left[\int_\Delta\frac{\log\dot\mu(x)}{x-z}\frac{p(x)\dd x}{w_+(x)} - \sum_{i=1}^g\int_{b_i}^{a_{i+1}}\frac{2\pi\ic c_{\dot\mu,i}}{y-z}\frac{p(y)\dd y}{w(y)}\right]\right\},
\end{equation}
\( z\in\overline\C\setminus[a_1,b_{g+1}] \), where the constants \( c_{\dot\mu,i} \) are defined by
\begin{equation}
\label{cmui}
c_{\dot\mu,i} := \frac1{2\pi\ic}\int_\Delta\log\dot\mu(x)\frac{l_i(x)\dd x}{w_+(x)}.
\end{equation}

\begin{proposition}
\label{prop:szego}
The Szeg\H{o} function \( S_{\dot\mu}(z) \) is analytic and non-vanishing in its domain of definition. Moreover, it holds that
\begin{equation}
\label{szego-jump}
\left\{
\begin{array}{ll}
S_{\dot\mu+}(x) = S_{\dot\mu-}(x)e^{-2\pi\ic c_{\dot\mu,k}}, & x\in(b_k,a_{k+1}), \quad k\in\{1,\ldots,g\}, \medskip \\
|S_{\dot\mu\pm}(x)|^2  = \dot\mu(x), & x\in \cup_{i=1}^g(a_i,b_i).
\end{array}
\right.
\end{equation}
The function \( S_{\dot\mu}^2(z) \) does not depend on the choice of a polynomial \( p(z) \).
\end{proposition}
\begin{proof}
Observe that
\begin{equation}
\label{van0}
x^j - \sum_{i=1}^g l_i(x)\int_{b_i}^{a_{i+1}}\frac{y^j\dd y}{w(y)} \equiv 0
\end{equation}
for any \( j\in\{0,\ldots,g-1\} \). Indeed, \eqref{van0} holds with \( x^j,y^j \) replaced by \( l_k(x),l_k(y) \) for any \( k\in\{1,\ldots,g\} \) by the very definition of the polynomials \( l_k(x) \) in \eqref{li}. Since these polynomials are linearly independent, \eqref{van0} follows. Given \( p(z) \) as above,  \eqref{van0} yields that
\begin{equation}
\label{van1}
\frac{p(x)}{x-z_l} - \sum_{i=1}^g l_i(x) \int_{b_i}^{a_{i+1}}\frac{p(y)}{y-z_l}\frac{\dd y}{w(y)} \equiv 0
\end{equation}
for each \( l\in\{1,\ldots,d\} \), and that
\begin{equation}
\label{van2}
\frac{p(x)}{x-z} - \sum_{i=1}^g l_i(x)\int_{b_i}^{a_{i+1}}\frac{p(y)}{y-z}\frac{\dd y}{w(y)} = \mathcal O\left(z^{d-g-1}\right), \quad z\to\infty.
\end{equation}
The first claim of the proposition holds because the term in square brackets in \eqref{szego} vanishes at every zero \( z_l \) of \( p(z) \) by \eqref{van1} and at infinity with order \( g+1-d \) by \eqref{van2} (notice that \( (w/p)(z)=z^{g+1-d}+\mathcal O(z^{g-d})\) as \( z\to\infty \)).  

The well-known boundary behavior of Cauchy integrals, see \cite[Section~I.4]{Gakhov}, yields the first relation in \eqref{szego-jump}. It also follows from \cite[Section~I.4]{Gakhov} that \( S_{\dot\mu+}(x)S_{\dot\mu-}(x)=\dot\mu(x) \). If the polynomial \( p(z) \) has real coefficients, then the obvious conjugate-symmetry implies the second relation in \eqref{szego-jump}. To prove it in general, take a ratio of Szeg\H{o} functions corresponding to two different polynomials (one having real coefficients). Call it \( S(z) \) and lift it to \( \boldsymbol D \). Lift \( S^{-1}(z) \) to \( \boldsymbol D^* \). Since \( S_+S_-\equiv1 \) on \( \Delta \), the lifted function is in fact holomorphic on the whole surface \( \RS \). That is, it is a constant. Because both, \( S(z) \) and \( S^{-1}(z) \), are equal to this constant, the constant is either \( 1 \) or \(-1\). Hence, \( S_{\dot\mu}^2(z) \) is independent of \( p(z) \) and, as relations \eqref{szego-jump} do not change if \( S_{\dot\mu}(z) \) is multiplied by \( -1 \), the second relation in \eqref{szego-jump} follows.
\end{proof}

In \cite[Section~5.2]{Y18}, the above construction was applied directly to the function \( \rho(x) \), see~\eqref{smooth-meas1}--\eqref{smooth-meas2}. However, we use it only for \( \dot\mu(x) \) since we can then leverage the positivity of \( \dot\mu(x) \), it makes the comparison with the case of classical Szeg\H{o} functions (\( g=0 \)) clearer, and because a Szeg\H{o} function of a polynomial can be constructed differently using Riemann theta functions, see~Section~\ref{sec:3.8}. There it will become clear that \eqref{szego} is not the only function which is non-vanishing and has boundary behavior as in \eqref{szego-jump}. However, the ratio of any two such functions is equal to the exponential of a linear combination of holomorphic differentials with coefficients that are integer multiples of \( 2\pi\ic \).

\subsection{Jacobi Inversion Problem}

Recall that a divisor on \( \RS \) is a formal linear combination of points from \( \RS \) with integer coefficients. Denote by \( \mathrm{Jac}(\RS) \) the Jacobi variety of \( \RS \), that is, the set of equivalence classes \( [\vec u] \), \( \vec u\in\C^g \), where \( [\vec u]=[\vec v]\) if and only if \( \vec u - \vec v = \vec j + \mathsf B\vec m \) for some \( \vec j,\vec m\in\Z^g \) and \( \mathsf B \) was defined in \eqref{B}. Abel's map from the divisors of \( \RS \) onto \( \mathrm{Jac}(\RS) \) is defined by
\begin{equation}
\label{am}
\sum n_j\z_j \mapsto \left[\sum \int_{\boldsymbol b_{g+1}}^{\z_j} \vec{\mathcal H}\right], 
\end{equation}
where \( \vec{\mathcal H} := (\mathcal H_1,\ldots,\mathcal H_g)^\mathsf{T} \) is the column vector of the holomorphic differentials on \( \RS \), see \eqref{hol-diff} (since the difference of two paths with the same endpoints is homologous to a linear combination of the cycles of the homology basis with integer coefficients, this map is indeed well-defined). 

Let \( \vec\omega_n:=(\omega_{n,1},\ldots,\omega_{n,g})^\mathsf{T} \) and \( \vec c_{\dot\mu}:=(c_{\dot\mu,1},\ldots,c_{\dot\mu,g})^\mathsf{T} \) be the column vectors of real constants defined in \eqref{omegank} and \eqref{cmui}, respectively. Further, let \( \s_i\in\boldsymbol D^* \) be such that \( \pi(\s_i)\in (b_i,a_{i+1}) \) and is a zero of the polynomial \( m(x) \) from \eqref{smooth-meas2}, \( i\in\{1,\ldots,g\} \). We are interested in the solutions of the following Jacobi inversion problem: \emph{find a divisor \( \mathcal D_n = \sum_{i=1}^g \x_{n,i} \) such that}
\begin{equation}
\label{main-jip}
\left[\sum_{i=1}^g \int_{\boldsymbol b_{g+1}}^{\x_{n,i}} \vec{\mathcal H}\right] = \left[\sum_{i=1}^g \int_{\boldsymbol b_{g+1}}^{\s_i} \vec{\mathcal H} + \vec c_{\dot\mu} + \vec\omega_n\right].
\end{equation}
It is known that \eqref{main-jip} is always solvable and the solution is unique up to a principal divisor (divisor of a rational function on \( \RS \)). That is, if \( \mathcal D_n - \big\{\mbox{ principal divisor }\big\} \) is an effective divisor (all the coefficients are positive), then it also solves \eqref{main-jip}. Immediately one can see that the subtracted principal divisor should have a positive part of degree at most \( g \). As \( \RS \) is hyperelliptic, such divisors come solely from rational functions on \( \overline\C \) lifted to \( \RS \) \cite[Appendix 1]{NutS77}. In particular, such principal divisors are involution-symmetric. Hence, if a solution of \eqref{main-jip}  contains at least one involution-symmetric pair of points, then replacing this pair by another such pair produces a different solution of \eqref{main-jip}. However, if a solution does not contain such a pair, then it solves \eqref{main-jip} uniquely. Hence, in the case where  \eqref{main-jip} has multiple solutions, we denote by \( \mathcal D_n \) the one whose involution-symmetric pairs are all of the form \( \boldsymbol\infty+\boldsymbol\infty^* \), where \(\boldsymbol\infty \) and \( \boldsymbol\infty^* \) are the points on top of infinity in \( \boldsymbol D \) and \( \boldsymbol D^* \), respectively. In particular, if \( k_n \) is the number of involution-symmetric pairs within a given solution of \eqref{main-jip}, then \( \mathcal D_n - k_n\boldsymbol\infty-k_n\boldsymbol\infty^* \) is a common part of every solution of \eqref{main-jip}.

\begin{proposition}
\label{prop:jip}
With an appropriate labeling, solution \( \mathcal D_n = \sum_{i=1}^g \x_{n,i} \) of \eqref{main-jip} is such that \( \pi(\x_{n,i}) \in [b_i,a_{i+1}] \), \( i\in\{1,\ldots,g\} \) (in particular, there are no other solutions). 
\end{proposition}
\begin{proof}
Recall that \( \pi(\s_i)\in(b_i,a_{i+1}) \), \( i\in\{1,\ldots,g\} \), and observe that \eqref{main-jip} can be rewritten as
\begin{equation}
\label{jip}
\left[\sum_{i=1}^g \int_{\boldsymbol b_i}^{\x_{n,i}}\vec{\mathcal H}\right] = \left[\vec V_n\right], \quad \vec V_n := \vec V + \vec c_{\dot\mu} + \vec\omega_n, \quad \vec V:=\sum_{i=1}^g \int_{[\boldsymbol b_i,\s_i]}\vec{\mathcal H},
\end{equation}
where the path of integration \( [\boldsymbol b_i,\s_i] \) is a ``segment'' on \( \RS \), that is, \( \pi:[\boldsymbol b_i,\s_i]\to[b_i,s_i] \) is a bijection. In this case \eqref{hol-diff} immediately shows that the vector \( \vec V \) and therefore the vectors \( \vec V_n \) have real entries. Denote by \( \bar \z \) the point on the same sheet of \( \RS \) as \( \z \) with \( \pi(\bar\z) = \bar z \) if \( \z\not\in\bd \) and \( \bar\z=\z^* \) if \( \z\in\bd \). Since the polynomials \( l_i(x) \) have real coefficients, it holds that
\[
\int_{\boldsymbol b_i}^{\z}\vec{\mathcal H} = \overline{\int_{\boldsymbol b_i}^{\bar\z}\vec{\mathcal H}}
\]
by \eqref{hol-diff}, where the paths of integration are reflections of each other under the map \( \s \mapsto \bar\s \). Thus, since \( \vec V_n \) is a real vector, if \( \mathcal D_n = \sum \x_{n,i} \) solves \eqref{jip}, so does \( \sum \bar\x_{n,i} \). As explained just before the proposition, it must holds that \( \mathcal D_n=\sum\bar \x_{n,i} \). Now, let us partially fix the labeling of the points of \( \mathcal D_n \). Namely, if a cycle \( \ualpha_k \) contains at least one point of the divisor \( \mathcal D_n \), we label one of these points as \( \x_{n,k} \) and distribute indices \( i \) arbitrarily to the rest of elements of \( \mathcal D_n \). Further, let \( I_1\cup I_2\cup I_3\cup I_4 := \{1,\ldots,g\} \) be disjoint sets such that \( \x_{n,i}\in\ualpha_i \) for \( i\in I_1 \), \( \pi(\x_{n,i})\in\overline\R\setminus[a_1,b_{g+1}] \) for \( i\in I_2 \), \( \x_{n,i}\in\ualpha_j \) for some \( j\neq i \) when \( i\in I_3 \) (notice that necessarily \( j\in I_1 \) in this case), and for each \( i\in I_4 \) there exists \( j\in I_4 \) such that \( \x_{n,j} = \bar \x_{n,i} \). Then \( \vec V_{n,i} := \int_{\boldsymbol b_i}^{\x_{n,i}}\vec{\mathcal H}\in\R^g \) when \( i\in I_1 \). Moreover,
\[
\left[\int_{\boldsymbol b_i}^{\x_{n,i}}\vec{\mathcal H}\right] = \left[\vec V_{n,i} + \frac12\big[\mathsf B\big]_i\right], \quad \vec V_{n,i} := \int_{\boldsymbol a_1}^{\x_{n,i}}\vec{\mathcal H}\in\R^g,
\]
for \( i\in I_2 \), where \( \big[\mathsf B\big]_i \) is the \( i \)-th column of \( \mathsf B \) and the path of integration in the last integral has a natural projection that belongs to \( \overline\R\setminus[a_1,b_{g+1}] \) (it is allowed to pass through points on top of infinity). Further, 
\[
\left[\int_{\boldsymbol b_i}^{\x_{n,i}}\vec{\mathcal H}\right] = \left[\vec V_{n,i} + \frac12\big[\mathsf B\big]_i+\frac12\big[\mathsf B\big]_j\right], \quad \vec V_{n,i} := \int_{[\boldsymbol b_j,\x_{n,i}]}\vec{\mathcal H}\in\R^g,
\]
for \( i\in I_3 \), where \( j \) is such that \( \x_{n,i}\in \ualpha_j \) and one needs to notice that adding rather than subtracting half a column of \( \mathsf B \) does not change the point  on the Jacobi variety as vectors that differ by a linear combination of the columns of \( \mathsf B \) with integer coefficients define the same element of \( \mathrm{Jac}(\RS) \). Next, let \( i,j\in I_4 \) be such that \( \x_{n,i} = \bar \x_{n,j} \). Then
\[
\left[\int_{\boldsymbol b_i}^{\x_{n,i}}\vec{\mathcal H} + \int_{\boldsymbol b_i}^{\x_{n,j}}\vec{\mathcal H}\right] = \left[\vec V_{n,i} + \frac12\big[\mathsf B\big]_i + \frac12\big[\mathsf B\big]_j\right], \quad \vec V_{n,i} = 2\Re\left(\int_{\boldsymbol b_i}^{\x_{n,i}}\vec{\mathcal H}\right).
\]
Denote by \( d_i+1 \) the number of elements of \( \mathcal D_n \) that belong to \( \ualpha_i \) for \( i\in I_1 \), and let \( I_5 \) be the largest subset of \( I_4 \) such that for any \( i\in I_5 \) there exists \( j \in I_4\setminus I_5 \) for which \(\x_{n,i} = \bar \x_{n,j} \) and \(j<i \). Then it follows from \eqref{jip} that
\[
\left[\vec V_n\right] = \left[ \sum_{i\in I_1\cup I_2\cup I_3\cup I_5 } \vec V_{n,i} + \vec U \right], \quad \vec U := \frac12\left(\sum_{i\in I_1}d_i\big[\mathsf B\big]_i + \sum_{i\in I_2\cup I_3\cup I_4}\big[\mathsf B\big]_i \right).
\]
Recall that the entries of \( \mathsf B \) are purely imaginary and that the columns of \( \mathsf B \) are linearly independent. Thus, \( \vec U\in (\ic\R)^g \) and \( [\vec U] = [\vec0 ] \) only if the sets \( I_2=I_3=I_4 =\varnothing \), which then  yields that each \( d_i=0 \). This observation finishes the proof of the proposition.
\end{proof}

Proposition~\ref{prop:jip} did not appear in \cite{Y18} since there are no reasons to believe that solutions of \eqref{main-jip} are confined to a certain subset of  \( \RS^g \) for more general geometries. 

\subsection{Riemann Theta Function}
Let \( \am(\z) \) be a function defined by
\begin{equation}
\label{af}
\am(\z) = \int_{\boldsymbol b_{g+1}}^{\z} \vec{\mathcal H}, \quad \z\in\RS_{\ualpha,\ubeta}:=\RS\setminus\cup_{i=1}^g\{\ualpha_i\cup\ubeta_i\},
\end{equation}
where the path of integration can be any as the domain of definition is simply connected. This function has continuous traces on the cycles of the homology basis (away from the points of intersection of different cycles) that satisfy
\begin{equation}
\label{abel-jump}
\am_+(\s) - \am_-(\s) = \left\{
\begin{array}{rl}
-\mathsf B \vec e_k, & \s\in\ualpha_k, \medskip \\
\vec e_k, & \s\in\ubeta_k,
\end{array}
\right. \quad k\in\{1,\ldots,g\},
\end{equation}
by \eqref{B} and the normalization of \( \vec{\mathcal H} \). 

The theta function associated with \( \mathsf B \) is an entire transcendental function of \( g \) complex variables defined by
\[
\theta\left(\vec u\right) := \sum_{\vec n\in\Z^g}\exp\bigg\{\pi\mathrm{i}\vec n^\mathsf{T}\mathsf B\vec n + 2\pi\mathrm{i}\vec n^\mathsf{T}\vec u\bigg\}, \quad \vec u\in\C^g.
\]
As shown by Riemann, the symmetry of \( \mathsf B \) and positive definiteness of its imaginary part ensures the convergence of the series for any \( \vec u \). It can be directly checked that \( \theta(-\vec u)=\theta(\vec u) \) and it enjoys the following periodicity property:
\begin{equation}
\label{theta-periods}
\theta\left(\vec u + \vec j + \mathsf B\vec m\right) = \exp\bigg\{-\pi \mathrm{i}\vec m^\mathsf{T}\mathsf B \vec m - 2\pi \mathrm{i}\vec m^\mathsf{T}\vec u\bigg\}\theta\big(\vec u\big), \quad \vec j,\vec m\in\Z^g.
\end{equation}
It is also known that \( \theta\left(\vec u\right)=0 \) if and only if \( \big[\vec u-\vec K\big]\) is the image of some  effective divisor of degree \( g-1 \) under Abel's map \eqref{am}, where \( \vec K \) is a fixed  vector known as the vector of Riemann constants (it can be explicitly defined via \( \vec{\mathcal H} \)).

Let \( \vec V=:(V_1,\ldots,V_g)^\mathsf{T} \) and \( \vec V_n=:(V_{n,1},\ldots,V_{n,g})^\mathsf{T} \) be given by \eqref{jip}. Set
\begin{equation}
\label{thetan}
\Theta_n(\z) := \frac{\theta\left(\am(\z) - \vec V_n - \vec K\right)}{\theta\left(\am(\z) - \vec V - \vec K\right)},
\end{equation}
which is a multiplicatively multi-valued meromorphic function on \( \RS \) with a simple zero at each \( \x_{n,i} \), a simple pole at each \( \s_i \), \( i\in\{1,\ldots,g\} \), and otherwise non-vanishing and finite. In fact, it follows from \eqref{abel-jump} and \eqref{theta-periods} that it is holomorphic, non-vanishing, and single-valued in \( \RS_{\ualpha}:=\RS\setminus\cup_{i=1}^g\ualpha_i \) and for \( \s\in\ualpha_k \) it holds that
\begin{equation}
\label{jump2}
\Theta_{n+}(\s)  = \Theta_{n-}(\s) \exp\left\{2\pi\ic\big(V_k - V_{n,k}\big)\right\} = \Theta_{n-}(\s) \exp\left\{-2\pi\ic(c_{\dot\mu,k} +\omega_{n,k})\right\}
\end{equation}
(these traces do vanish at \( \x_{n,i} \)'s and blow up at \( \s_i \)'s). 

The functions \( \Theta_n(\z) \) form a normal family in \( \RS_{\ualpha} \). Indeed, for any subsequence \( \{ n_k \}_k \) there always exists a subsequence \( \{ n_{k_i} \}_i \) along which the vectors \( \vec V_{n_{k_i}} \) converge to some vector \( \vec V_* \) since these vectors form a bounded subset of \( \R^g \) by \eqref{omegank} and \eqref{jip}. Then the corresponding limit point is obtained be replacing \( \vec V_n \) with \( \vec V_* \) in \eqref{thetan} and is holomorphic and non-vanishing in \( \RS_{\ualpha} \) (its zeroes must belong to the \( \ualpha \)-cycles). Hence, according to Montel's theorem, the functions \( \Theta_n(\z) \) are uniformly bounded on closed subsets of \( \RS_{\ualpha} \). Moreover, since their zeroes belong to the \( \ualpha \)-cycles, the above argument also shows that the reciprocals of the functions \( \Theta_n(\z) \) also form a normal family in \( \RS_{\ualpha} \) and therefore the functions themselves are uniformly bounded away from zero on closed subsets of \( \RS_{\ualpha} \). 

Recall the notation \( \bar\z \) introduced after \eqref{jip}. There exists a constant \( c_n \) such that
\begin{equation}
\label{flip-theta}
\overline{\Theta_n(\bar\z)} = c_n\Theta_n(\z), \quad \z\in\RS_{\ualpha}.
\end{equation}
Indeed, functions \( \Theta_n(\z) \) and \( \overline{\Theta_n(\bar\z)} \) are both holomorphic in \( \RS_{\ualpha} \), have the same jumps across the \( \ualpha \)-cycles (the fact that \( \vec c_{\dot\mu}+\vec\omega_n \in\R^g\) is important here), and have the same poles and zeroes (points on the \( \ualpha \)-cycles remain unchanged under the map \(\s\mapsto\bar\s\)). Thus, their ratio is holomorphic on the whole surface \( \RS \) and therefore is a constant.

\begin{proposition}
\label{prop:Tn}
For each natural number \( n \), \( z\in\overline\C\setminus[a_1,b_{g+1}] \), and \( \z\in \boldsymbol D\), let
\begin{equation}
\label{Tn}
T_n(z) := \Theta_n(\z^*)\Theta_n^{-1}(\z).
\end{equation}
Then \( T_n(z) \) is holomorphic and non-vanishing in its domain of definition and
\begin{equation}
\label{Tn-bdry}
\left\{
\begin{array}{ll}
T_{n+}(x) = T_{n-}(x)e^{4\pi\ic(c_{\dot\mu,k} +\omega_{n,k})}, & x\in(b_k,a_{k+1}), \quad k\in\{1,\ldots,g\}, \medskip \\
|T_{n\pm}(x)| \equiv 1, & x\in\Delta.
\end{array}
\right.
\end{equation}
Hence, \( T_n(z) \) can be analytically continued through each gap \( (b_i,a_{i+1}) \) and any such continuation has a simple pole at \( s_i \), a simple pole at \( x_{n,i} \) if \( \x_{n,i}\in \boldsymbol D \), and a simple zero at \( x_{n,i} \) if \( \x_{n,i}\in\boldsymbol D^* \) (if \( \x_{n,i} \) is a ramification point of \( \RS \) then \( T_n(z) \) has a non-zero finite limit at \( x_{n,i} \)).   Moreover, the families \( \big\{T_n(z)\big\}_n \) and \( \big\{T_n(z)^{-1}\big\}_n \) are normal in  \( \overline\C\setminus[a_1,b_{g+1}] \).
\end{proposition}
\begin{proof}
The properties of the functions \( T_n(z) \) follow immediately from the corresponding properties of the functions \( \Theta_n(\z) \) except for the second identity of \eqref{Tn-bdry}. To derive it, notice that on the one hand \( T_n(z) = \overline{T_n(\bar z)}\) by \eqref{flip-theta}  and therefore \( T_{n\pm}(x) =\overline{T_{n\mp}(x)}\) for \( x\in\Delta \), and on the other, \( T_{n\pm}(x) = 1/T_{n\mp}(x) \), \( x\in\Delta \), by the very definition \eqref{Tn}. These two equalities together do yield the desired claim. 
\end{proof}

\subsection{Szeg\H{o} Functions of Polynomials}
\label{sec:3.8}
Take \( q(z)=z-e \), \( e\not\in\Delta \). As usual, let \( \e\in\boldsymbol D^* \) be such that \( \pi(\e)=e \) and let \( \am(\e) \) be given by \eqref{af}, where with a slight abuse of notation we set \( \am(\e) := \am_+(\e) \) if \( \e\in\cup_{i=1}^g\ualpha_i \). Define 
\[
S_e(\z) := \frac{\theta\left(\am(\z) - \am(\boldsymbol\infty^*) - \vec K\right)}{\theta\left(\am(\z) - \am(\e) - \vec K\right)} \left\{
\begin{array}{rl}
1, & \z\in\boldsymbol D, \smallskip \\
q(z), & \z\in\boldsymbol D^*,
\end{array}
\right.
\]
where, as before, \( \boldsymbol\infty \) and \( \boldsymbol\infty^* \) are the points on top infinity that belong to \( \boldsymbol D \) and \( \boldsymbol D^* \), respectively. This function is holomorphic and non-vanishing on \( \RS_{\ualpha}\setminus\bd \) (notice that both the numerator and the denominator of the fraction above have a zero of order \( g-1 \) at \( \boldsymbol b_{g+1} \) as can be easily seen from \eqref{am} and \eqref{af}), whose traces on \(\bd \) and the \( \ualpha \)-cycles satisfy 
\[
S_{e+}(\s)=S_{e-}(\s)
\left\{
\begin{array}{rl}
q(s), & \s\in\bd, \smallskip \\
e^{2\pi\ic (\am_k(\e)-\am_k(\boldsymbol\infty^*))}, & \s\in\ualpha_k, \;\; k\in\{1,\ldots,g\},
\end{array}
\right.
\]
where \( \am_k(\z) \) is the \( k \)-th component of the vector \( \am(\z) \) and we exclude the points of intersection of different cycles. It is quite straightforward to see that the product \( S_e(\z)S_e(\z^*) \) is continuous across both \( \bd \) and the \( \ualpha \)-cycles and therefore is an entire function on the whole surface. Thus, \( S_e(\z)S_e(\z^*) \equiv S_e^2(\boldsymbol a_1) \) on \( \RS \). Hence, if we put \( S_q(z) := S_e(\z)/S_e(\boldsymbol a_1) \), \( \z\in\boldsymbol D \), then \( S_q(z) \) is holomorphic and non-vanishing in \( \overline\C\setminus[a_1,b_{g+1}] \) and satisfies
\[
\left\{
\begin{array}{ll}
S_{q+}(x)S_{q-}(x)=q(x), & x\in\Delta, \smallskip \\
S_{q+}(x)=S_{q-}(x)e^{2\pi\ic (\am_k(\e)-\am_k(\boldsymbol\infty^*))}, & x\in(b_k,a_{k+1}), \;\; k\in\{1,\ldots,g\}.
\end{array}
\right.
\]
Now, if \( e\in(-\infty,\infty)\setminus[a_1,b_{g+1}] \), one can see from \eqref{hol-diff} and \eqref{af} that the constant \( \am_k(\e)-\am_k(\boldsymbol\infty^*) \) is real. In this case, similarly to \eqref{flip-theta}, one can check that  \( \overline{S_e(\bar\z)} = c_e S_e(\z) \) and therefore \( S_e(\z)/S_e(\boldsymbol a_1) = \overline{S_e(\bar\z)}/\overline{S_e(\boldsymbol a_1)} \). Hence, it additionally holds that \( |S_{q\pm}(x)|^2=q(x)\) for \( x\in\Delta \). In the same way one can also check that the last equality also holds when \( q(z) = (z-e)(z-\bar e) \) (in this case we define \( S_q(z) \) as a product of \( S_{\cdot-e}(z) \) and \( S_{\cdot-\bar e}(z) \)).

The non-uniqueness of the above construction stems from the vectors \( \am(\boldsymbol\infty^*) \) and \( \am(\e) \) as they can be replaced by any pair \( \vec u \) and \( \vec v \) such that \( [\am(\boldsymbol\infty^*)] = [\vec u] \) and \( [\am(\e)]=[\vec v] \). Such a substitution will lead to a different function \( S_q(z) \) with the same jump relation on \( \Delta \), but different constant jumps in the gaps. Clearly, the ratio of two such Szeg\H{o} functions is equal to a constant multiple of \( \exp\{2\pi\ic {\vec m}^{\mathsf T} \am(\z)\} \) for some \( \vec m\in\Z^g \) as can be seen from \eqref{theta-periods}.

The above construction also leads to a  formula for \( S_n(\z) \) different from \eqref{Sn}. Indeed, up to multiplication by a constant multiple of \( \exp\{2\pi\ic {\vec m}^{\mathsf T} \am(\z)\} \), \( \vec m\in\Z^g \), \( S_n(\z) \) is equal to
\[
\left(\prod_{e\in E_n,~|e|<\infty}S_e(\z)\right)\left(\frac{\theta\left(\am(\z) - \am(\boldsymbol\infty^*) - \vec K\right)}{\theta\left(\am(\z) - \am(\boldsymbol\infty) - \vec K\right)}\right)^n.
\]

\subsection{Main Theorem}

Recall \eqref{w} and Propositions~\ref{prop:psin}--\ref{prop:Tn}. An analog of Theorem~\ref{thm:M1-MP} in the multi-cut case can be stated as follows.

\begin{theorem}
\label{thm:Mg-MP}
Let \( f(z) \) be given by \eqref{markov} and \eqref{smooth-meas1}--\eqref{sobolev}. Further, let \( p_n(z)/q_n(z) \) be the multipoint Pad\'e approximant of \( f(z) \) associated with \( E_n \), where the interpolation sets \( E_n \) are conjugate-symmetric and there exists a neighborhood of \( [a_1,b_{g+1}]\) disjoint from all the multisets \( E_n \). Then it holds that
\[
f(z)-\frac{p_n(z)}{q_n(z)} = (2+o(1))(T_n\psi_n)(z)\frac{\big(mS_{\dot\mu}^2\big)(z)}{w(z)},
\]
as \( n\to \infty \) locally uniformly in \( \overline\C\setminus[a_1,b_{g+1}] \). 
\end{theorem}
Recall that \( \psi_n(z) \) satisfies \( |\psi_n(z)|<1 \), \( z\in\overline\C\setminus\Delta \), and that it has \( 2n \) zeroes there. Since these zeroes are separated from \( \Delta \) for all \( n \), the functions \( \psi_n(z) \) converge to zero geometrically fast in \( \overline\C\setminus\Delta \). Recall also that the functions \( T_n(z) \) are uniformly bounded away from infinity and zero on closed subsets of \( \overline\C\setminus[a_1,b_{g+1}] \).
\begin{proof}
Let \( S_n(\z) \), \( S_{\dot\mu}(z) \), and \( \Theta_n(\z) \) be given by \eqref{Sn}, \eqref{szego}, and \eqref{thetan}, respectively. Define
\begin{equation}
\label{Psin}
\Psi_n(\z) := (S_n\Theta_n)(\z) \left\{
\begin{array}{rl}
1/S_{\dot\mu}(z), & \z\in\boldsymbol D, \smallskip \\
(mS_{\dot\mu})(z), & \z\in\boldsymbol D^*.
\end{array}
\right.
\end{equation}
It holds that \( \Psi_n(\z) \) is a meromorphic function in \( \RS\setminus\bd \) with a pole of order \( n \) at \( \boldsymbol\infty \), a zero or order \( n-g \) at \( \boldsymbol\infty^* \), a simple zero at each \( \x_{n,i} \), \( i\in\{1,\ldots,g\} \), and otherwise non-vanishing and finite. Moreover, it has continuous traces on \( \bd \) that satisfy
\begin{equation}
\label{Psin-jump}
\Psi_{n+}(\s) = \Psi_{n-}(\s)(v_n/\rho)(x), \quad \s\in\bd,
\end{equation}
by \eqref{smooth-meas2}, \eqref{Sn-jump},  \eqref{szego-jump}, and \eqref{jump2}. Hence, \( \Psi_n(\z) \) is exactly the function from \cite[Proposition~3.3]{Y18}. Further, let \( \vec W_n := \vec V_n + \am(\boldsymbol\infty^*) - \am(\boldsymbol\infty) \), see \eqref{jip}. Define
\begin{equation}
\label{upsilon}
\Upsilon_n(\z) := \frac{\theta\left(\am(\z) - \am(\boldsymbol\infty) - \vec K\right)}{\theta\left(\am(\z) - \am(\boldsymbol\infty^*) - \vec K\right)}\frac{\theta\left(\am(\z) - \vec W_n - \vec K\right)}{\theta\left(\am(\z) - \vec V_n - \vec K\right)}.
\end{equation}
Then it follows from \eqref{jip} and \eqref{theta-periods} that \( \Upsilon_n(\z) \) is a rational function on \( \RS \) with  poles only at \( \x_{n,1},\ldots,\x_{n,g},\boldsymbol\infty^* \), which are simple, and a zero at \( \boldsymbol\infty \) (the rest of \( g \) zeroes are uniquely determined by the corresponding Jacobi inversion problem). This is exactly the function from \cite[Proposition~3.4]{Y18}. Observe also that an argument virtually similar to the one after \eqref{jump2} shows that the functions \( \Upsilon_n(\z) \) form a normal family in \( \RS_{\ualpha}\setminus\{\boldsymbol\infty^*\} \).

Let \( \gamma_n \) be a constant such that \( \lim_{\z\to\boldsymbol\infty}\gamma_nz^{-n}\Psi_n(\z) = 1 \). Recall also the definition of  \( R_n(z) \) in \eqref{Rn}. It was shown in \cite[Theorem~3.7]{Y18} that
\begin{equation}
\label{asymp1}
\left\{
\begin{array}{rcl}
q_n(z) &=& \gamma_n \big(1+o(1)+o(1)\Upsilon_n(\z)\big)\Psi_n(\z), \medskip \\
(wR_n)(z) &=& \gamma_n\big(2+o(1)+o(1)\Upsilon_n(\z^*)\big)\Psi_n(\z^*),
\end{array}
\right.
\end{equation}
locally uniformly in \( \overline\C\setminus\Delta \), where \( \z\in\boldsymbol D \) and the error terms \( o(1) \) vanish at infinity. The normality of \( \Upsilon_n(\z) \) and the fact that each of them has a simple pole at \( \boldsymbol\infty^* \) while the error functions are holomorphic and vanish at infinity coupled with the maximum modulus principle show that we can write
\begin{equation}
\label{asymp2}
\left\{
\begin{array}{rcl}
q_n(z) &=& \gamma_n \big(1+o(1)\big) \Psi_n(\z), \medskip \\
(wR_n)(z) &=& \gamma_n\big(2+o(1)\big)\Psi_n(\z^*),
\end{array}
\right.
\end{equation}
locally uniformly in \( \overline\C\setminus[a_1,b_{g+1}] \). The desired claim now follows from the very definitions of \( R_n(z) \) and \( \Psi_n(\z) \), \eqref{psin-sn}, and \eqref{Tn}.
\end{proof}

\section{\( \bar H_\R^2 \)-approximation of Markov Functions On Several Intervals}

Assume now in addition to \eqref{smooth-meas1}--\eqref{sobolev} that \( \Delta\subset(-1,1) \).

\subsection{Auxiliary Polynomials}
Given \( w(z) \) as in \eqref{w}, set
\begin{equation}
\label{w-tilde}
\tilde w(z) := z^{g+1}w(1/z),
\end{equation}
which is holomorphic in \( \C\setminus\Delta^{-1} \) and normalized so that \( \tilde w (0) =1 \) (in particular, \( \tilde w(x)>0 \) for \( x\in(-1,1) \)). Let \( \mathsf W \) be the following matrix:
\begin{equation}
\label{matW}
\mathsf W := \left[\int_{b_i}^{a_{i+1}}\big(x+x^{-1}\big)^l\frac{x^g\dd x}{(w\tilde w)(x)}\right]_{i=1,l=1}^{g,g},
\end{equation}
where \( i \) is the row index and \( l \) is the column one (powers \( l \) do go up to \( g \) but exclude \( 0 \)). Let us abbreviate \( y_l=x_l+x_l^{-1} \). Then
\[
\det(\mathsf W) = \int_{b_1}^{a_2}\cdots\int_{b_g}^{a_{g+1}} V\big(y_1,\ldots,y_g\big)\frac{(y_g\cdots y_1)(x_g^g\cdots x_1^g)}{(w\tilde w)(x_g)\cdots (w\tilde w)(x_1)}\dd x_g\cdots \dd x_1,
\]
where, as before,  \( V(y_1,\ldots,y_g) \) is the Vandermonde determinant.

If \( 0\in \Delta \), then it is clear that \( \det(\mathsf W)\neq 0 \) since all the gaps \( (b_k,a_{k+1}) \) are disjoint and the Jukovski map \( x+x^{-1}\) preserves this property. In this case we can set
\begin{equation}
\label{elli1}
\ell_j(x) := x^g\sum_{i=1}^g \ell_{ij} \big( x+x^{-1} \big)^i = h_j \prod_{i=1}^g(x-x_{ij})(1-xx_{ij}), 
\end{equation}
where \( [\ell_{ij}]_{i,j=1}^g = \mathsf W^{-1} \) and we agree that \( |x_{ij}|\leq1\). Then these polynomials satisfy
\begin{equation}
\label{elli2}
\int_{b_k}^{a_{k+1}}\frac{\ell_j(x)\dd x}{(w\tilde w)(x)} = \delta_{kj}, \quad k,j\in\{1,\ldots,g\}.
\end{equation}
Since each \( \ell_j(x) \) has real coefficients, it has at least one zero in each gap \( (b_k,a_{k+1}) \), \( k\neq j \). 

To see that polynomials satisfying \eqref{elli2} exist even if \( 0 \not\in\Delta \), observe that they can be carried forward and backward by M\"obius transformations. Indeed, let \( \ell(t)=h\prod_{i=1}^g(t-t_i)(1-tt_i) \) be a symmetric polynomial of degree \( 2g \) and \( t(x) = (x-x_0)/(1-xx_0) \) be a M\"obius transformation that carries a system of intervals \( \Delta_1 \) into another system \( \Delta_2 \). Then
\[
\ell(t) = h\prod_{i=1}^g\left(\frac{1-x_0^2}{1-x_ix_0}\right)^2\frac{\prod_{i=1}^g(x-x_i)(1-xx_i)}{(1-xx_0)^{2g}},
\]
where \( t_i=t(x_i) \). Let \( w_{\Delta_i}(x) \) and \( \tilde w_{\Delta_i}(x) \) be defined via \eqref{w} and \eqref{w-tilde}, respectively, with respect to the system of intervals \( \Delta_i \). Then
\[
(w_{\Delta_2}\tilde w_{\Delta_2})(t) = \prod_{k=1}^{g+1}\frac{(1-x_0^2)^2}{(1-a_kx_0)(1-b_kx_0)}\frac{(w_{\Delta_1}\tilde w_{\Delta_1})(x)}{(1-xx_0)^{2g+2}}.
\]
Since \( \dd t = (1-x_0^2)(1-xx_0)^{-2}\dd x \), it holds that
\[
\int_{t(c)}^{t(d)}\frac{\ell(t)\dd t}{(w_{\Delta_2}\tilde w_{\Delta_2})(t)} = \int_c^d\frac{\ell_*(x)\dd x}{(w_{\Delta_1}\tilde w_{\Delta_1})(x)},
\]
where \( \ell_*(x) := h_* \prod_{i=1}^g(x-x_i)(1-xx_i) \) and \( h_* := h (1-x_0^2)\prod_{k=1}^{g+1} \frac{(1-a_kx_0)(1-b_kx_0)}{(1-x_{k-1}x_0)^2} \). Then, by applying  to \( \Delta \) a M\"obius transformation that sends a point from \( \Delta \) to the origin, constructing polynomials satisfying \eqref{elli1} and \eqref{elli2} for the image of \( \Delta \), and carrying them back by the inverse M\"obius transformation, we shall construct the desired polynomials \( \ell_i(x) \) for \( \Delta \).

\subsection{Condenser Map}

Let \( u(x) \) be a symmetric polynomial of degree \( 2g \) such that
\begin{equation}
\label{phi-norm}
\int_\Delta \frac{u(x)\dd x}{(w_+\tilde w)(x)} = -\ic \qandq \int_{b_k}^{a_{k+1}}\frac{u(x)\dd x}{(w\tilde w)(x)} = 0, \;\; k\in\{1,\ldots,g\}.
\end{equation}
It can be readily checked that
\begin{equation}
\label{U}
u(x) = u_0\left(x^g - \sum_{i=1}^g\left(\int_{b_i}^{a_{i+1}}\frac{y^g\dd y}{(w\tilde w)(y)}\right)\ell_i(x)\right)  = c\prod_{i=1}^g(x-x_i)(1-x_ix),
\end{equation}
where the constant \( c \) (and therefore \( u_0 \)) is chosen so that the first condition in \eqref{phi-norm} is fulfilled and we agree that \( |x_i|\leq 1 \), \( i\in\{1,\ldots,g\} \). Observe that each gap \( (b_k,a_{k+1}) \) must contain exactly one zero of of \( u(x) \), say \( x_k \). In particular, since \( w_+(x)=(-1)^{g+1-k}\ic |w(x)| \) for \( x\in(a_k,b_k) \), it holds that \( (-1)^{g+1-k}u(x)>0 \) for \( x\in(a_k,b_k) \), \( k\in\{1,\ldots,g+1\} \). 

\begin{proposition}
\label{prop:varphi}
Define
\begin{equation}
\label{varphi}
\varphi(z) := \exp\left\{\pi\int_1^z\frac{u(s)\dd s}{(w\tilde w)(s)}\right\},
\end{equation}
where the path of integration lies entirely in \( \overline\C\setminus\big([a_1,b_{g+1}]\cup[a_1,b_{g+1}]^{-1}\big) \). The function \( \varphi(z) \) is well-defined and holomorphic in its domain of definition and satisfies
\begin{equation}
\label{varphi-sym}
\varphi(\bar z) = \overline{\varphi(z)} \qandq \varphi(1/z) = 1/\varphi(z).
\end{equation}
Moreover, the increment of the argument of \( \varphi(z) \) along the unit circle is equal to \( 2\pi \) and it holds that
\begin{equation}
\label{varphi-prop}
\left\{
\begin{array}{ll}
|\varphi(\tau)|\equiv1, & \tau\in\T, \smallskip \\
|\varphi(x)|\equiv\rho^{\pm1}, & x\in\Delta^{\pm1}, \;\; \rho:=\varphi(b_{g+1})<1.
\end{array}
\right.
\end{equation}
Furthermore, \( \varphi(z) \) has continuous traces in the gaps \( \cup_{i=1}^g(b_i,a_{i+1}) \) that satisfy
\begin{equation}
\label{varphi-jump}
\varphi_+(x) = \varphi_-(x)e^{-2\pi\ic\omega_k}, \quad \omega_k := \sum_{i=1}^k\int_{a_i}^{b_i}\dd\omega_{\Delta,\T},
\end{equation}
for \( x\in(b_k,a_{k+1}) \) and each \( k\in\{1,\ldots,g\} \), where the probability measure \( \omega_{\Delta,\T} \) is given by
\begin{equation}
\label{conden-eq}
\dd\omega_{\Delta,\T}(x) := \frac{|u(x)|\dd x}{|(w\tilde w)(x)|} = \frac{\ic u(x)\dd x}{(w_+\tilde w)(x)} , \quad x\in\Delta.
\end{equation}
\end{proposition}
\begin{proof}
The function \( \varphi(z) \) is well-defined due to the first requirement in \eqref{phi-norm}. Relations \eqref{varphi-sym} follow from the conjugate-symmetry of the integrand as well as the symmetry with respect to \( s\mapsto 1/s \) (here, one needs to include the differential \( \dd s \) as well). The first relation in \eqref{varphi-prop} is a direct consequence of \eqref{varphi-sym} and to get the second relation one needs to use the second requirement in \eqref{phi-norm} and the observation that the integrand in \eqref{varphi} is purely imaginary on each side of \( \Delta \). The claim about the increment of the argument follows directly from \eqref{phi-norm}. Relations \eqref{varphi-jump} can be verified via a direct computation.
\end{proof}

The probability measure \( \omega_{\Delta,\T} \) can be interpreted from the point of view of potential theory as the condenser equilibrium distribution on \( \Delta \) of the condenser \( (\Delta,\T) \). Moreover, since the traces \( \varphi_\pm(x) \) have constant argument in each gap \( (b_i,a_{i+1}) \), \( \varphi(z) \) is a conformal map of \( \overline\C\setminus\big([a_1,b_{g+1}]\cup[a_1,b_{g+1}]^{-1}\big) \) into an annulus \( \{\rho<|z|<\rho^{-1}\} \) with ``internal spikes''.

\subsection{Condenser Szeg\H{o} Function} 

Set
\begin{equation}
\label{K}
K(z;x) := \frac1{x-z}+\frac x{1-xz} = \frac{1-2xz+x^2}{(x-z)(1-xz)}
\end{equation}
to be the condenser version of the Cauchy kernel. Given a H\"older-smooth function \( \lambda(x) \) on \( \Delta \) that is positive away from the endpoints \( \{a_i,b_i\}_{i=1}^{g+1} \) and has integrable logarithm, let
\begin{multline}
\label{Dmu}
D_\lambda(z) := \exp\left\{\frac1{2\pi\ic}\frac{(w\tilde w)(z)}{u(z)}\left[ \int_\Delta K(z;x)\log\left(\frac{\lambda(x)}{G_\lambda}\right)\frac{u(x)\dd x}{(w_+\tilde w)(x)} - \right.\right. \\ \left.\left.  \sum_{i=1}^g\int_{b_i}^{a_{i+1}}2\pi\ic \kappa_{\lambda,i} K(z;y)\frac{u(y)\dd y}{(w\tilde w)(y)}\right]\right\}
\end{multline}
for \( z\in\overline\C\setminus\big([a_1,b_{g+1}]\cup[a_1,b_{g+1}]^{-1}\big) \), where the constants \( G_\lambda \) and \( k_{\lambda,i} \) are given by
\begin{equation}
\label{Gmu}
G_\lambda := \exp\left\{\int\log\lambda(x)\dd\omega_{\Delta,\T}(x)\right\} \qandq \kappa_{\lambda,i} := \frac1{2\pi\ic}\int_\Delta\log\left(\frac{\lambda(x)}{G_\lambda}\right) \frac{\ell_i(x)\dd x}{(w_+\tilde w)(x)}.
\end{equation}

\begin{proposition}
\label{prop:CS}
It holds that \( D_\lambda(z) \) is a well-defined non-vanishing holomorphic function in its domain of definition. It has continuous traces on each open subinterval of \( \Delta \) and in each gap that satisfy
\begin{equation}
\label{conden-jump}
\left\{
\begin{array}{ll}
G_\lambda|D_{\lambda\pm}(x)|^2  = \lambda(x), & x\in \cup_{i=1}^g(a_i,b_i), \medskip \\
D_{\lambda+}(x) = D_{\lambda-}(x)e^{-2\pi\ic \kappa_{\lambda,k}}, & x\in(b_k,a_{k+1}), \quad k\in\{1,\ldots,g\}.
\end{array}
\right.
\end{equation}
Furthermore, the increment of the argument of \( D_\lambda(z) \) along the unit circle is equal to zero and it holds that
\begin{equation}
\label{conden-sym}
D_\lambda(1/z) = 1/D_\lambda(z) \qandq |D_\lambda(\tau)| \equiv 1, \;\; \tau\in\T.
\end{equation}
Moreover, if \( \lambda(x) = |x-e|^\alpha\lambda_*(x) \), where \( e\in\{a_i,b_i\}_{i=1}^{g+1} \) and \( \lambda_*(x) \) is H\"older-smooth, then
\begin{equation}
\label{conden-van}
|D_\lambda(z)|^2 = |z-e|^\alpha |D_*(z)|^2,
\end{equation}
where the function \( D_*(z) \) is bounded and non-vanishing in some neighborhood of \( e \) (including the traces on \( \Delta \)).
\end{proposition}
\begin{proof}
To prove the first claim, it is enough to investigate what happens in the unit disk only as \( D_\lambda(1/z) = 1/D_\lambda(z) \). Indeed, we have that \( (w\tilde w/u)(1/z) = (1/z^2)(w\tilde w/u)(z) \) and \( (1/z^2)K(1/z;x) = (2/z) - K(z;x) \). Since the integrals in \eqref{Dmu} with \( K(z;x) \) replaced by \(  1/z \) vanish due to \eqref{phi-norm}, \eqref{conden-eq}, and the choice of \( G_\lambda \) in \eqref{Gmu}, the claim follows. 

To continue, let us make the following observation: it follows from \eqref{elli1} and \eqref{elli2} that
\[
\ell_k(x) - \sum_{i=1}^g \ell_i(x)\int_{b_i}^{a_{i+1}} \frac{\ell_k(y)\dd y}{(w\tilde w)(y)} \equiv 0
\]
for each \( k\in\{1,\ldots,g \} \). Further, let \( p(x) \) be any symmetric polynomial of degree \( 2g \), which we can write as, \( p(x) = p_0 x^g + \sum_{k=1}^g c_k(p) \ell_k(x) \) for some constants \( c_k(p) \). The previous identity then yields that
\begin{equation}
\label{van3}
p(x) - \sum_{i=1}^g \ell_i(x)\int_{b_i}^{a_{i+1}} \frac{p(y)\dd y}{(w\tilde w)(y)} = p_0\left( x^g - \sum_{i=1}^g \ell_i(x)\int_{b_i}^{a_{i+1}} \frac{y^g\dd y}{(w\tilde w)(y)} \right) = \frac{p_0}{u_0} u(x),
\end{equation}
where \( u(x) \) was defined in \eqref{U}. 

Going back to the proof of the first claim, observe that \( K(z;x) - (z-x)^{-1} \) is an analytic function of \( z\in\D \) for each \( x\in\Delta \), see \eqref{K}. Therefore, \( D_\lambda(z) \) is analytic in \( \D\setminus[a_1,b_{g+1}] \) away from the zeroes of \( u(x) \). However, it is analytic at those zeroes as well since the term in square brackets vanishes at them. Indeed, let \( x_j \) be such a zero. Observe that \( u(x)K(x_j;x) \) is a symmetric polynomial of degree \( 2g \) according to \eqref{K}. Therefore,  the term in square brackets is equal to
\begin{multline*}
\int_\Delta \left(u(x)K(x_j;x)-\sum_{i=1}^g \ell_i(x)\int_{b_i}^{a_{i+1}}\frac{u(y)K(x_j;y)\dd y}{(w\tilde w)(y)}\right)\log\left(\frac{\lambda(x)}{G_\lambda}\right)\frac{\dd x}{(w_+\tilde w)(x)}  = \\ \frac{h_j}{u_0}\int_\Delta u(x)\log\left(\frac{\lambda(x)}{G_\lambda}\right)\frac{\dd x}{(w_+\tilde w)(x)} = 0
\end{multline*}
by \eqref{van3} and the choice of \( G_\lambda \) in \eqref{Gmu}, where \( h_j \) is the coefficient next to \( x^g \) of \( u(x)K(x_j;x) \). Hence, \( D_\lambda(z) \) is indeed holomorphic and non-vanishing in \( \overline\C\setminus\big([a_1,b_{g+1}]\cup[a_1,b_{g+1}]^{-1}\big) \).

Relations \eqref{conden-jump} and \eqref{conden-van} hold due to known behavior of Cauchy integrals and the fact that \( K(z;x) \) is ``essentially'' a Cauchy kernel, \cite[Sections~4.2 and~8.6]{Gakhov}. 

The first identity in \eqref{conden-sym} has already been explained and the second one follows from the first and the fact that \( D_\lambda(\bar z) = \overline{D_\lambda(z)} \). 

It only remains to prove the claim about the increment of the argument of \( D_\lambda(z) \), say \( I \), along \( \T \) (the proof is an adjustment of the one in \cite[Lemmas~2.40 and~5.2]{BStW01}).  \( I \) is equal to the integral of the tangential derivative of the argument of \( D_\lambda(z) \) over the unit circle (oriented counter-clockwise). Due to Cauchy-Riemann relations we have that  \( I = \frac1{2\pi}\oint_\T (\partial_{\boldsymbol n} h)(s)|\dd s|\), where \( \partial_{\boldsymbol n} \) is the partial derivative on \( \partial U \) with respect to the inner normal and \( h(z) := \log|D_\lambda(z)| \) is harmonic in \( U:=\D\setminus\Delta \) (harmonicity across the gaps follows from the second relation in \eqref{conden-jump} and the fact that constants \( \kappa_{\lambda,i} \) are real). Then it follows from Green's formula that
\[
\oint_{\partial U} g(s)(\partial_{\boldsymbol n} h)(s)|\dd s| = \oint_{\partial U} h(s)(\partial_{\boldsymbol n} g)(s)|\dd s|,
\]
where \( g(z) := \log|\varphi(z)/\rho| \) is also a harmonic function in \( U \), see \eqref{varphi} and \eqref{varphi-prop}. Since \( h(x)=(1/2)\log(\lambda/G_\lambda)(x) \) for \( x\in\Delta \) and \( h(z)\equiv 0 \) on \( \T \) by \eqref{conden-jump} and \eqref{conden-sym}, and \( g(z)\equiv0 \) on \( \Delta \) and \( g(z)\equiv-\log\rho \) on \( \T \) by \eqref{varphi-prop}, it holds that
\[
I = \frac{-1}{4\pi\log\rho}\int_\Delta \log\left(\frac{\lambda(x)}{G_\lambda}\right)\big(\partial_{\boldsymbol n+}g+\partial_{\boldsymbol n-}g)(x)\dd x.
\]
Now, it readily follows from \eqref{varphi} that
\[
(\partial_{\boldsymbol n\pm}g)(x) = \left.\pm\pi \partial_y \Re\left(\int_1^z\frac{u(s)\dd s}{(w\tilde w)(s)}\right)\right|_{z=x} =  \frac{\pi\ic u(x)}{(w_+ \tilde w)(x)}.
\]
Plugging the above expression in the last formula for \( I \) and recalling \eqref{conden-eq} gives us
\[
I = \frac{-1}{2\log\rho}\int_\Delta\log\left(\frac{\lambda(x)}{G_\lambda}\right)\dd\omega_{\Delta,\T}(x) = 0, 
\]
where we used \eqref{Gmu} for the last equality.
\end{proof}

\subsection{Main Theorem}

Let \( f(z) \) be given by \eqref{markov} and \eqref{smooth-meas1}--\eqref{sobolev} and \( \{r_n(z)\} \) be a sequence of irreducible critical points in rational \( \bar H_\R^2 \)-approximation of \( f(z) \). Write \( r_n(z) = (p_n/q_n)(z) \). Recall that \( q_n(z) \) satisfies \eqref{ortho} with 
\begin{equation}
\label{vn-cp}
v_n(z) = \varkappa_n\tilde q_n^2(z),
\end{equation}
where \( \tilde q_n(z) = z^nq_n(1/z) \) and \( \varkappa_n \) is the reciprocal of the square of the product of non-zero zeroes of \( q_n(z) \), which turns \( v_n(z) \) into a monic polynomial. Let divisors \( \mathcal D_n = \sum_{i=1}^g\x_{n,i} \) be the solutions of Jacobi inversion problem \eqref{main-jip} with \( v_n(z) \) given by \eqref{vn-cp}. Recall that the points \( \x_{n,i} \) could be labeled so that \( x_{n,k} \in [b_k,a_{k+1}] \) for each \( k\in\{1,\ldots,g\} \). Set
\begin{equation}
\label{mn}
m_n(z) := (z-x_{n,1})\cdots(z-x_{n,g}).
\end{equation}
Denote further by \( B_n(z) \) the Blaschke product that vanishes exactly at those points \( x_{n,i} \) that correspond to the elements of the divisor \( \mathcal D_n \) that belong to \( \boldsymbol D \), that is,
\begin{equation}
\label{Bn}
B_n(z) := \prod_{\x_{n,i}\in\boldsymbol D}\frac{z-x_{n,i}}{1-x_{n,i}z}.
\end{equation}
We shall denote by \( d_n \) the number of zeroes of \( B_n(z) \), i.e., the number of elements of \( \mathcal D_n \) that belong to \( \boldsymbol D \). Finally, put
\begin{equation}
\label{lambdan}
\lambda_n(x) := \rho(x)B_n^2(x)/m_n(x), \quad x\in\Delta,
\end{equation}
which is a positive and continuous function (with the exception of those \( x_{n,i} \) that belong to \( \{a_i,b_i\}_{i=1}^{g+1} \)). Recall the definition of the function \( w(z) \) in \eqref{w} as well as Propositions~\ref{prop:varphi}--\ref{prop:CS}. An analog of Theorem~\ref{thm:M1-CP} in the multi-cut case can be stated as follows.

\begin{theorem}
\label{thm:Mg-CP}
Let \( f(z) \) be given by \eqref{markov} and \eqref{smooth-meas1}--\eqref{sobolev}.  Further, let \( \{r_n(z)\} \) be a sequence of irreducible critical points in rational \( \bar H_\R^2 \)-approximation of \( f(z) \). Then it holds that
\[
f(z)-r_n(z) = \big(2G_{\lambda_n}+o(1)\big)\frac{m_n(z)}{B_n^2(z)}\frac{D_{\lambda_n}^2(z)}{w(z)}\left(\frac\rho{\varphi(z)}\right)^{2(n-d_n)}
\]
as \( n\to \infty \) locally uniformly in \( \overline\C\setminus\big([a_1,b_{g+1}]\cup[a_1,b_{g+1}]^{-1}\big) \).
\end{theorem}

It is tempting to surmise that \( D_{\lambda_n}(z)/\varphi^{n-d_n}(z) \) is analytic across the gaps the way the product \( (T_n\psi_nS_{\dot\mu}^2)(z) \) from Theorem~\ref{thm:Mg-MP} is. Unfortunately, while the choice of \( T_n(z) \) was driven exactly by the need to cancel the jumps of \( (\psi_nS_{\dot\mu}^2)(z) \) across the gaps, the function \( \lambda_n(x) \) was devised out of necessities of the forthcoming proof. Observe also that its definition requires the knowledge of the divisors \( \mathcal D_n \) and therefore of the exact locations of the zeroes of the polynomials \( q_n(z) \). Of course, it would be much preferable to be able to argue that points \( x_{n,i} \) need to be selected so that \( D_{\lambda_n}(z)/\varphi^{n-d_n}(z) \) has no jumps across the gaps and then to prove Theorem~\ref{thm:Mg-CP} with this knowledge only.

Below, we follow the approach developed in \cite{BStW01} for the proof of Theorem~\ref{thm:M1-CP}.

\begin{proof}[Proof of Theorem~\ref{thm:Mg-CP}]
As above, write \( r_n(z)=p_n(z)/q_n(z) \). Since \( q_n(z) \) has real coefficients and its zeroes belong to the unit disk, \( v_n(z) \) has real coefficients and is non-zero in \( \D \). Thus, Theorem~\ref{thm:Mg-MP} is applicable to \( (p_n/q_n)(z) \) and we have
\begin{equation}
\label{error1}
(f-r_n)(z) = (2+o(1))(T_n\psi_n)(z)\frac{\big(mS_{\dot\mu}^2\big)(z)}{w(z)}
\end{equation}
locally uniformly in \( \overline\C\setminus[a_1,b_{g+1}]\). Define \( h_n(z) := \log|H_n(z)| \), where
\[
H_n(z) := (T_n\psi_n)(z)\frac{m(z)}{m_n(z)}\frac{B_n^2(z)S_{\dot\mu}^2(z)}{G_{\lambda_n}D_{\lambda_n}^2(z)}\left( \frac{\varphi(z)}\rho\right)^{2(n-d_n)}.
\]
It follows from \eqref{varphi-jump} and \eqref{conden-jump} that \( |H_{n+}(x)|=|H_{n-}(x)| \) for \( x\in \cup_{k=1}^g(b_k,a_{k+1}) \) and the jump is constant in each gap. Moreover, according to Proposition~\ref{prop:Tn} and equations \eqref{conden-van}, \eqref{Bn}, the function \( |(mT_nB_n^2/m_nD_{\lambda_n}^2)(z)| \) is non-vanishing in the gaps of \( \Delta \) (observe that if \( x_{n,i} \) is equal to either \( b_i \) or \( a_{i+1} \), then it is of course a simple zero of \( m_n(z) \), it is essentially a "simple pole" of \( D_{\lambda_n}^2(z) \), while all other functions are non-vanishing around it). Since the zeroes of the polynomials \( q_n(z) \) belong to \( \Delta \), the zeroes of \( \tilde q_n(z) \) and therefore of \( \psi_n(z) \) belong to \( \Delta^{-1} \). Thus, \( h_n(z) \) is harmonic in \( \D\setminus\Delta \) as well as across \( \T \) (recall that \( \Delta\subset(-1,1) \)). Furthermore, since
\[
G_{\lambda_n}|D_{\lambda_n\pm}^2(x)| = \lambda_n(x) = B_n^2(x)\dot\mu(x)m(x)/m_n(x)
\]
by \eqref{conden-jump}, it holds by \eqref{psin-bdry}, \eqref{szego-jump}, \eqref{Tn-bdry}, and \eqref{varphi-prop} that
\begin{equation}
\label{hnDelta}
h_n(x) \equiv 0,  \quad x\in\Delta.
\end{equation}

To describe the behavior of \( h_n(z) \) on the unit circle, set \( b_n(z) := q_n(z)/\tilde q_n(z) \). Then it follows from \eqref{Rn} that
\begin{equation}
\label{cp-err}
(f-r_n)(z) = \varkappa_n \big(b_n^{-2}q_nR_n\big)(z).
\end{equation}
Let now \( \Psi_n(\z) \) and \( \gamma_n \) be as in the proof of Theorem~\ref{thm:Mg-MP}, see \eqref{Psin} and \eqref{asymp1}. Then we get from \eqref{asymp2} that
\[
(f-r_n)(z) = \big(\varkappa_n\gamma_n^2\big)\frac{2 + o(1)}{w(z)}  \frac{\Psi_n(\z)\Psi_n(\z^*)}{b_n^2(z)}
\]
locally uniformly in \( \overline\C\setminus[a_1,b_{g+1}] \). It is quite easy to see from \eqref{Psin-jump} that the product \( \Psi_n(\z)\Psi_n(\z^*) \) has no jump on \( \bd \) and therefore is a rational function on \( \RS \). Since it is symmetric with respect to the involution \( \z\mapsto\z^* \), it must be a lift of a rational function on \( \overline\C \) to \( \RS \). The form of its zero/pole divisor then yields that
\begin{equation}
\label{prod-Psis}
\Psi_n(\z)\Psi_n(\z^*) = (\gamma_n^*/\gamma_n)m_n(z),
\end{equation}
where \( \gamma_n^* := \lim_{\z\to\boldsymbol\infty^*}\Psi_n(\z)z^{n-g} \). Since Blaschke products \( b_n(z) \) are unimodular on the unit circle, we can write
\begin{equation}
\label{error2}
|(f-r_n)(\tau)| = (2 + o(1))|\varkappa_n\gamma_n\gamma_n^*||(m_n/w)(\tau)|
\end{equation}
uniformly for \( \tau\in\T \). Comparing \eqref{error1} and \eqref{error2}, we see that
\[
\varkappa_n\gamma_n\gamma_n^* = (1+o(1))|(T_n\psi_n)(\tau)||(m/m_n)(\tau)||S_{\dot\mu}^2(\tau)|
\]
uniformly for \( \tau\in\T \). We do not need an absolute value around the constant \( \varkappa_n\gamma_n\gamma_n^* \) since
\[
\varkappa_n\gamma_n\gamma_n^* = \varkappa_n\lim_{\x\to\boldsymbol\infty}z^{2n-g}\Psi_n(\z^*)/\Psi_n(\z) = \lim_{0<x\to\infty}x^{2n-g}\frac{\psi_n(x)}{\tilde q_n^2(x)}\frac{T_n(x)m(x)}{S_{\dot\mu}^2(x)}>0,
\]
where the first equality follows straight from the definitions of these constants, see \eqref{asymp1} and \eqref{prod-Psis}, the second one is a consequence of \eqref{Psin}, \eqref{psin-sn}, \eqref{Tn}, \eqref{vn-cp}, as well as the meromorphy of \( \Psi_n(\z^*)/\Psi_n(\z) \) around \( \boldsymbol\infty \), and the last inequality holds because \( (\tilde q_nS_{\dot\mu})(x) \) is real for \( x>b_{g+1} \), \( \psi_n(x) \) is positive there (exponential of a real-valued function), and so is \( T_n(x) \) (it is real and non-vanishing there and \( \lim_{x\to b_{g+1}^+}T_n(x)=1\)). Because Blaschke products  \( B_n(z) \) are unimodular on \( \T \), it follows from \eqref{varphi-prop} and \eqref{conden-sym} that
\begin{equation}
\label{hnT}
h_n(\tau) = \log\left(\frac{\varkappa_n\gamma_n\gamma_n^*}{G_{\lambda_n}\rho^{2(n-d_n)}}\right) +o(1)
\end{equation}
uniformly on \( \T \). Let us now recall \cite[Lemma~4.7]{BStW01}. It states that if \( U\subset \overline\C \) is a domain such that \( \partial U = K_1\cup K_2 \), where \( K_1,K_2 \) are two compact disjoint sets and if \( h(z) \) is a harmonic function in \( U \) such that
\begin{equation}
\label{o-int}
\oint_C \partial_{\boldsymbol n} h(s)|\dd s| =0,
\end{equation}
where \( \partial_{\boldsymbol n} \) is the normal derivative on a chain \( C \) of smooth Jordan curves that separates \( K_1 \) from \( K_2 \) and has winding number \( 1 \)  with respect to \( K_1 \) and \( 0 \) with respect to \( K_2 \), then for both \( l\in\{1,2\} \) it holds that
\[
\sup_{z^\prime\in K_l} \limsup_{U\ni z\to z^\prime} h(z) \geq \inf_{z^\prime\in K_{3-l}} \liminf_{U\ni z\to z^\prime} h(z).
\]
We would like to apply this lemma with \( U=\D\setminus\Delta \), \( K_1=\Delta \), \( K_2=\T \), and \( h(z)=h_n(z) \). Since \( h_n(z) \) continuously extends to \( \T \) as well as both sides of \( \Delta \) where it satisfies \eqref{hnDelta}, the lemma yields that
\begin{equation}
\label{hnT1}
\max_{\tau\in\T}|h_n(\tau)|\geq 0 \geq \min_{\tau\in\T}|h_n(\tau)|,
\end{equation}
granted integral condition \eqref{o-int} is satisfied. Take \( C \) to be a circle in \( \D\setminus\Delta \) centered at the origin. Since \( h_n(z) \) is harmonic across \( \T \), it follows from Green's formula that if \eqref{o-int} is satisfied on \( \T \), then it is satisfied on \( C \) as just described. Hence, we can, in fact, take \( C=\T \). The integral of the normal derivative of \( h_n(z) \) over \( \T \) is equal to the integral of the tangential derivative of the harmonic conjugate of \( h_n(z) \), that is, it is equal to the total increment of the argument of \( H_n(z) \) along \( \T \). The increment of the argument of \( T_n(z)(m/m_n)(z)S_{\dot\mu}^2(z) \) is equal to \( 0 \) by the argument principle since this is a non-vanishing holomorphic function in \( \{|z|>\max\{|a_1|,|b_{g+1}|\}\} \). The increment of the argument of \( D_{\lambda_n}(z) \) was shown to be equal to zero in Propositions~\ref{prop:CS}. Finally, the increment of the argument of \( \psi_n(z)B_n^2(z)\varphi^{2(n-d_n)}(z) \) is zero because the first factor has exactly \( 2n \) zeroes counting multiplicities all belonging to \( \Delta^{-1} \) (thus, the increment of its argument is \( -4\pi n \)), the Blaschke product \( B_n^2(z) \) has exactly \( 2d_n \) zeroes in \( \D \) (so, the increment of its argument is \( 4\pi d_n \)), and the increment of the argument of \( \varphi(z) \) is \( 2\pi \), see Proposition~\ref{prop:varphi}. Hence, \eqref{hnT1} does indeed hold and \eqref{hnT} can be improved to
\begin{equation}
\label{hnT2}
h_n(\tau) = \log\left(\frac{\varkappa_n\gamma_n\gamma_n^*}{G_{\lambda_n}\rho^{2(n-d_n)}}\right) +o(1) = o(1).
\end{equation}

Combining the last estimate with \eqref{hnDelta}, we get from the maximum principle for harmonic functions that \( h_n(z)= o(1) \) uniformly in \( \overline\D \). Thus, the functions \( H_n(z) \) form a normal family in \( \D\setminus[a_1,b_{g+1}] \) and every limit point of this family is a unimodular constant. It readily follows from \eqref{psin}, \eqref{szego}--\eqref{cmui}, \eqref{flip-theta}--\eqref{Tn}, \eqref{varphi}, and \eqref{K}--\eqref{Gmu} (an explanation similar to the one preceding \eqref{hnT}) that \( H_n(x)>0 \) for \( x\in(b_{g+1},1) \). Therefore, the only limit point is the function \( 1 \), that is, 
\begin{equation}
\label{Hn-est1}
H_n(z)= 1 +o(1)
\end{equation}
locally uniformly in \( \D\setminus[a_1,b_{g+1}] \). It should be clear from the definition of \( H_n(z) \) that \eqref{Hn-est1} proves the theorem in \( \D\setminus[a_1,b_{g+1}] \).

Let's now consider what happens outside of the unit disk. Since \( b_n(1/z) = b_n^{-1}(z) \), it follows from \eqref{cp-err}, \eqref{asymp2}, and \eqref{prod-Psis} that
\begin{eqnarray*}
(f-r_n)(z) &=& \varkappa_n^2 \frac{(q_nR_n)(z)(q_nR_n)(1/z)}{(f-r_n)(1/z)} \\
& = & (\varkappa_n\gamma_n\gamma_n^*)^2\frac{4+o(1)}{(f-r_n)(1/z)} \frac{m_n(z)m_n(1/z)}{w(z)w(1/z)}
\end{eqnarray*}
locally uniformly in \( \overline\C\setminus(\overline\D\cup\Delta^{-1}) \). Because we have proven the theorem already for \( z\in\D\setminus\Delta \), we further get that
\begin{eqnarray*}
(f-r_n)(z) &=& (2+o(1))(\varkappa_n\gamma_n\gamma_n^*)^2 \frac{m_n(z)}{w(z)} \frac{B_n^2(1/z)}{G_{\lambda_n}D_{\lambda_n}^2(1/z)}\left(\frac{\varphi(1/z)}{\rho}\right)^{2(n-d_n)} \\
& = & (2G_{\lambda_n}+o(1))\frac{m_n(z)}{w(z)}\frac{B_n^2(1/z)}{D_{\lambda_n}^2(1/z)}\left(\rho\varphi(1/z)\right)^{2(n-d_n)}
\end{eqnarray*}
locally uniformly in \( \overline\C\setminus(\overline\D\cup\Delta^{-1}) \), where we used \eqref{hnT2} for the second equality. The desired estimate in  \( \overline\C\setminus(\overline\D\cup\Delta^{-1}) \) now follows from \eqref{varphi-sym}, \eqref{conden-sym}, and the symmetries of Blaschke products with real zeroes. The full statement of the theorem now follows from the maximum modulus principle that allows us to extend the error estimate across the unit circle.
\end{proof}


\begin{thebibliography}{100}

\bibitem{BStW01}
L.~Baratchart, H.~Stahl, and F.~Wielonsky.
\newblock Asymptotic error estimates for ${L}^2$ best rational approximants to
  {M}arkov functions.
\newblock {\em J. Approx. Theory}, 108(1):53--96, 2001.

\bibitem{CYL98}
B. de la Calle Ysern and G.~L\'opez Lagomasino.
\newblock Strong asymptotics for orthogonal polynomials with varying measures and Hermite-Pad\'e approximation.
\newblock {\em J. Comput. Appl. Math.}, 99(1--2):91--103, 1998.

\bibitem{FarkasKra}
H. Farkas and I. Kra.
\newblock {\em Riemann Surfaces}, volume~71 of {\em Graduate Texts in Mathematics}.
\newblock Springer-Verlag, New York, 1992.


\bibitem{Gakhov}
F.D. Gakhov.
\newblock {\em Boundary Value Problems}.
\newblock Dover Publications, Inc., New York, 1990.

\bibitem{Lev69}
A.L. Levin.
\newblock The distribution of poles of rational functions of best approximation
  and related questions.
\newblock {\em Math. USSR Sbornik}, 9(2):267--274, 1962.

\bibitem{NutS77}
J.~Nuttall and S.R. Singh.
\newblock Orthogonal polynomials and {P}ad\'e approximants associated with a
  system of arcs.
\newblock {\em J. Approx. Theory}, 21:1--42, 1977.

\bibitem{Ransford}
T.~Ransford.
\newblock {\em Potential Theory in the Complex Plane}, volume~28 of {\em London
  Mathematical Society Student Texts}.
\newblock Cambridge University Press, Cambridge, 1995.

\bibitem{St00}
H.~Stahl.
\newblock Strong asymptotics for orthogonal polynomials with varying weights.
\newblock {\em Acta Sci. Math. (Szeged)}, 66(1--2):147--192, 2000.

\bibitem{Totik}
V. Totik.
\newblock \emph{Weighted approximation with varying weights}.
\newblock Lecture Notes in Math., Vol. 1569, Springer-Verlag, Berlin, 1994.

\bibitem{Y18}
M.~Yattselev.
\newblock Symmetric contours and convergent interpolation.
\newblock \emph{J. Approx. Theory}, 225:76--105, 2018.

\end{thebibliography}
\end{document}